\documentclass[a4paper,11.5pt]{article}
\usepackage{tikz}
\usetikzlibrary{calc}
\usepackage{amssymb}
\usetikzlibrary{quotes,angles}
\usetikzlibrary{plotmarks}
\usetikzlibrary{arrows,shapes,positioning}
\usetikzlibrary{decorations.markings}
\usepackage{geometry}
\tikzstyle arrowstyle=[scale=2]
\tikzstyle directed=[postaction={decorate,decoration={markings,
    mark=at position .65 with {\arrow[arrowstyle]{stealth}}}}]
\tikzstyle reverse directed=[postaction={decorate,decoration={markings,
    mark=at position .65 with {\arrowreversed[arrowstyle]{stealth};}}}]
    \tikzstyle left directed=[postaction={decorate,decoration={markings,
    mark=at position -.62 with {\arrow[arrowstyle]{stealth}}}}]

\tikzstyle left reverse directed=[postaction={decorate,decoration={markings,
    mark=at position -.62 with {\arrowreversed[arrowstyle]{stealth};}}}]
\usepackage{indentfirst}
\usepackage[plainpages=false]{hyperref}
\usepackage{amsfonts,latexsym,rawfonts,amsmath,amssymb,amsthm,mathrsfs}
\usepackage{amsmath,amssymb,amsfonts,latexsym,lscape,rawfonts}

\usepackage[all]{xy}
\usepackage{eufrak}
\usepackage{makeidx}         
\usepackage{graphicx,psfrag}
\usepackage{array,tabularx}

\usepackage{setspace}

\newtheorem{thm}{Theorem}[section]
\newtheorem{cor}[thm]{Corollary}
\newtheorem{lem}[thm]{Lemma}
\newtheorem{clm}[thm]{Claim}
\newtheorem{prop}[thm]{Proposition}

\theoremstyle{remark}
\newtheorem{rmk}[thm]{Remark}

\theoremstyle{definition}
\newtheorem{Def}[thm]{Definition}                                        %

\def \C {\mathbb C}
\def \Z {\mathbb Z}
\def \R {\mathbb R}

\title{Moduli spaces of $G_{2}-$instantons and $Spin(7)-$instantons on product manifolds}

{\small{\author{Yuanqi Wang\thanks{The University of Kansas, Lawrence, KS, USA.\ yqwang@ku.edu; yqwang2541@gmail.com . The previous institute of the author is Simons Center for Geometry and Physics, Stony Brook University, NY, USA.}}}

\date{\vspace{-5ex}}

\begin{document}
\maketitle
\begin{abstract} Let  $X$ be a closed $6-$dimensional manifold with a half-closed $SU(3)-$structure. On the product manifold $X\times S^{1}$, with respect to the product $G_{2}-$structure and on a pullback vector bundle from $X$, we show that any $G_{2}-$instanton  is equivalent to a Hermitian Yang-Mills connection on $X$ via a ``broken gauge". This result reveals the topological type of  the moduli of  $G_{2}-$instantons on $X\times S^{1}$. In dimension $8$, similar result holds for moduli of  $Spin(7)-$instantons. A generalization and an example are  given. 
\end{abstract}

\section{Introduction}
\subsection{Motivation and Background}
Following the programs of Donaldson-Thomas \cite{DonaldsonThomas} and Donaldson-Segal \cite{DonaldsonSegal}, it is tempting to generalize the classical  gauge theory in dimensions $2,3,4$ to dimensions $6,7,8$. In the  classification of holonomy groups by Berger and Simon (\cite{Berger}, \cite{Simons}), these higher dimensions correspond to the special holonomy groups $SU(3)$,  $G_{2}$, and $Spin(7)$. Based on  the programs in \cite{DonaldsonThomas}  and     \cite{DonaldsonSegal}, Walpuski  \cite{Walpuski} and Joyce \cite{Joyce} discussed the possible enumerative invariant on  $7-$dimensional manifolds ``counting"   $G_{2}-$instantons. 

The purpose of this note is to completely classify all $G_{2}-$instantons  (not only those which are invariant under a group action),  of a pullback vector bundle on a trivial circle bundle over a fairly general $6-$dimensional base manifold. Our main result (Theorem \ref{Thm Rigidity} below) shows that any $G_{2}-$instanton in this product setting is equivalent to the pullback of a Hermitian Yang-Mills connection on the $6-$dimensional base via a ``broken gauge" (see Definition \ref{Def ab gauges and iso trivial} below). Similar results also hold for projective $G_{2}-$instantons and $Spin(7)-$instantons.

Before stating the main theorem, we set up our terminology  and illustrate our ideas  along the way. 

The definitions in section \ref{section Building blocks for the product manifolds}--- \ref{section Moduli spaces and their topology} are necessary for our main result. 

 Here is another way to describe our purpose: we  seek for a dimension reduction for moduli of $G_{2}-$instantons on a product manifold $X\times S^{1}$ (trivial circle bundle over $X$). We consider a  $6-$manifold which admits a half-closed $SU(3)-$structure in the  sense of Definition \ref{Def CY quadruple} below.  
 \subsection{Building blocks for the product manifolds\label{section Building blocks for the product manifolds}}

\begin{Def}\label{Def CY quadruple}
Given a  $6$-dimensional  manifold $X$, we say that  $(J, g_{X}, \omega, \Omega)$ is a $SU(3)-$structure   (cf. \cite[(3.1) and the enclosing section]{KMT})  if 
\begin{enumerate} \item $J$ is an almost complex structure, $\Omega$ is a nowhere vanishing $(3,0)-$form.
\item  $g_{X}$ is a Hermitian metric on $X$ i.e. $g_{X}(J\cdot,J\cdot)=g_{X}(\cdot,\cdot)$.  $\omega=g(J\cdot, \cdot)$  is the associated real positive $(1,1)-$form. 
\item $|\Omega|^{2}_{g_{X}}=8$ i.e. $\frac{\omega^{3}}{3!}=\frac{1}{4}Re\Omega\wedge Im\Omega$. 
\end{enumerate}
A $SU(3)-$structure is called half-closed if $dRe\Omega=0$. 
\end{Def}
Throughout, we understand $S^{1}$ as the smooth Riemannian manifold $\R/2\pi \Z$, so its length is $2\pi$. Our main theorem and the proof hold for an arbitrary positive length.  Let $t$ be the coordinate variable of  $\R$ such that $dt$ descends  to the smooth closed (but not exact) $1$-form on $S^{1}$. All manifolds, bundles, gauges, connections, sections etc are assumed to be smooth unless otherwise specified. 
\begin{rmk}\label{Rmk abundant SU(3) str} The half-closed condition is not restricted to $Re\Omega$.   Given a $SU(3)-$structure such that $Im\Omega$ is closed,  then $(\cdot,\cdot,\cdot,\sqrt{-1} \Omega)$ is  half-closed. Given $J,\ \Omega$ as in Definition \ref{Def CY quadruple}.1 such that  $dRe\Omega=0$, by Lemma \ref{lem unitary frame} below, there  exist  abundant Hermitian metrics  $g_{X}$  such that $(J,g_{X},\omega,\Omega)$ is half-closed. \end{rmk}
\begin{rmk}A  half-closed $SU(3)-$structure is  said to be Calabi-Yau if $J$ is integrable, $\omega$ is  closed (K\"ahler), and  $\Omega$ is holomorphic. Then the metric $g_{X}$ must be Ricci flat by the normalization in Definition \ref{Def CY quadruple}.3. 

Another class of half-closed $SU(3)-$structures consists of nearly-K\"ahler $6-$manifolds, including $S^{6}$, $S^{3}\times S^{3}$ etc (see \cite[3.2]{KMT} and \cite{Harland}). 
\end{rmk}
\subsection{$G_{2}$ and $Spin(7)-$structures}
Before defining $G_{2}$ and $Spin(7)-$instantons, we need to define $G_{2}$ and $Spin(7)-$structures. 
\begin{Def}Let $\R^{7}$ be the $7-$dimensional Euclidean vector space with the co-frame\\  $\{e^{i},\ 1\leq i\leq 7\}$. We define the Euclidean associative $3-$form as
\begin{equation}\phi_{Euc}=e^{127}+e^{347}+e^{567}+e^{135}-e^{146}-e^{236}-e^{245},\ \textrm{where}\ e^{ijk}\triangleq e^{i}\wedge e^{j} \wedge e^{k}.\end{equation}
 Given a $7-$manifold $M$, a $G_{2}-$structure $\phi$ is a smooth $3-$form such that at every point $p$, there exists a co-frame $e^{i},\ 1\leq i\leq 7$ such that $\phi(p)=\phi_{Euc}$. $\phi$ determines a Riemannian metric $g_{\phi}$ and an orientation. This orientation is associated to the volume form $e^{1234567}$.  We let $\psi\triangleq \star_{g_{\phi}}\phi$. 
 \end{Def}

Given a $6-$manifold $X$ with a $SU(3)-$structure $(J,g_{X},\omega,\Omega)$, on the $7-$manifold $M\times S^{1}$,  the $3-$form \begin{equation}\label{equ G2 structure on the product} \phi=dt\wedge \omega+Re\Omega\end{equation}
is a $G_{2}-$structure whose induced metric is the product  $g_{X}+dt\otimes dt$. 

Next, we define $Spin(7)-$structures. 
\begin{Def}Let $\R^{8}$ be the $8-$dimensional Euclidean vector space with the co-frame $e^{i}$,\\ $0\leq i\leq 7$. We define the Euclidean Cayley $4-$form as
\begin{equation} \Psi_{Euc}\triangleq e^{0}\wedge \phi_{Euc}+\psi_{Euc}. 
\end{equation}
Given an $8-$dimensional manifold $M^{8}$, a $Spin(7)-$structure $\Psi$ is a $4-$form such that at every point $p$, there exists a co-frame $e^{i},\ 0\leq i\leq 7$ such that $\Psi(p)=\Psi_{Euc}$.
\end{Def}

Let $M$ be a $7-$manifold with a $G_{2}-$structure $\phi$.  On the $8-$dimensional manifold $M\times S^{1}$,   the $4-$form 
\begin{equation}\label{equ Spin7 structure on the product} \Psi=dt\wedge \phi+\psi\end{equation} is a $Spin(7)-$structure. The induced metric is the product $g_{\phi}+dt\otimes dt$.  The orientation is defined by $dt\wedge \phi\wedge \psi$,  so $\Psi$ is self-dual. 

\subsection{Iso-trivial connections}
The following definition of Hermitian vector bundles is the foundation of our discussion. 
\begin{Def}\label{Def isotrivial} Let $Y$ be a closed $n-$dimensional smooth manifold. A  smooth complex vector bundle $E\rightarrow Y$ is called a Hermitian vector bundle, if it admits a Hermitian metric, and the following holds. 
\begin{itemize}\item There is a finite open cover of $Y$, and a unitary trivialization $s_{U}$ of $E$ on each open set $U$ in the cover.

\item Any trivialization can be extended to a larger open set containing the closure of its domain.
\item  Any transition function on a non-empty intersection
 is smooth, and  can be extended smoothly to a larger open set containing the closure of the intersection.\end{itemize}

Let $adE$ denote the bundle of skew-adjoint endomorphisms with respect to the Hermitian metric, and $EndE$ denote the usual endomorphism bundle. Associated to a Hermitian vector bundle $E$, both $adE$ and $EndE$ are still Hermitian vector bundles. 

A gauge is a unitary automorphism that preserves the Hermitian metric. Let $\mathfrak{G}$ denote  the space of all smooth gauges. Given a  connection $A$ and $u\in \mathfrak{G}$, we adopt the convention $u(A)=A+u^{-1}d_{A}u$ i.e. $d_{u(A)}\triangleq u^{-1}\cdot d_{A}\cdot u.$ Then
\begin{equation}\label{equ right multiplication of gauges}v[u(A)]=(uv)(A)\ i.e.\ \textrm{the gauge-action is a right multiplication}.\end{equation} 

\textbf{Convention}: unless otherwise specified, all gauges and connections  are assumed to be unitary.

In order to study whether a gauge on the manifold $Y\times (0,2\pi)$ extends to a gauge on $Y\times S^{1}$, we introduce the following definition. 
\begin{Def}\label{Def smooth periodicity}(Smooth periodicity) Let $\pi$ denote the projection from $Y\times S^{1}$ (or $Y\times I$ for any interval $I\subset \R$) to $Y$. A smooth section (or connection) $v$ of $\pi^{\star}E\rightarrow Y\times [0,2\pi]$ 
is said to be periodic if $v(0)=v(2\pi)$. It  is said to be smoothly periodic  if it extends to  a smooth section (connection) of $\pi^{\star}E\rightarrow Y\times S^{1}$. 
\end{Def}

 We are particularly interested in the irreducible connections defined as follows. Given a smooth connection $B$ on $E\rightarrow Y$,  we define the stabilizer group as
 \begin{equation}\label{equ def Stab} \Gamma_{B}\triangleq \{u\in\mathfrak{G}|\ d_{B}u=0\}.
\end{equation}
 $B$ is said to be irreducible if $\Gamma_{B}=Center[U(m)]$ [which is homeomorphic to $U(1)$ and  $S^{1}$]. Abusing notation, we still denote by $B$ the pullback of the connection $B$ on $Y$ to $\pi^{\star}E\rightarrow Y\times {S}^{1}$ (or $\pi^{\star}E\rightarrow Y\times [0,2\pi]$).

We now define the aforementioned ``broken gauges". In our terminology, on $Y\times S^{1}$, \textit{a ``broken gauge" in general is not a gauge}. 

 As we shall see below, a version of ``monodromy" phenomenon is implicitly contained in the definition and remarks for the ``broken gauges".
\begin{Def}\label{Def ab gauges and iso trivial} (Admissible broken gauges and iso-triviality)

Given a smooth connection $B$ on $Y$, a smooth gauge $u$ on $\pi^{\star}E\rightarrow Y\times [0,2\pi]$ (see Definition \ref{Def Cinfty top} below) is called a $B-$admissible broken gauge (or admissible broken gauge for short) if $u(0)=Id$, $u(2\pi)\in \Gamma_{B}$, and $\chi_{u}\triangleq u^{-1}\frac{du}{dt}$ is smoothly periodic. A smooth connection $A$ on  $\pi^{\star}E\rightarrow Y\times S^{1}$ is said to be iso-trivial with respect to $B$, if there exists a $B-$admissible broken gauge $u$  such that $A=u(B)$.

In practice, we abbreviate ``iso-trivial with respect to $B$" to  `` iso-trivial", because our notation $u(B)$ (or similar) and/or the context  should clarify what  the $B$ is.   \end{Def}

\end{Def}
We stress again that, on a product manifold $Y\times S^{1}$,  \textit{the notion of a $B-$admissible broken gauge is completely different from the notion of a gauge}.  A $B-$admissible  broken gauge is  a gauge only if it is periodic i.e.  $u(0)=u(2\pi)=Id$.
\begin{rmk}\label{rmk isotrivial connections are smooth}
Conversely, by the  criteria in Claim \ref{clm periodic gauges} below, we routinely verify that for any connection $B$ on $E\rightarrow Y$, and  a $B$-admissible broken gauge $u$, $u(B)$ is a 
smooth connection on $\pi^{\star}E\rightarrow Y\times S^{1}$.\end{rmk}
\begin{rmk} Iso-triviality   is preserved by gauge-transformations on $Y\times S^{1}$. Please see Proposition \ref{Prop gauge equivalence characterization} below on gauge equivalence of iso-trivial connections.  \end{rmk}

\subsection{Hermitian Yang-Mills connections, $G_{2}-$instantons, and\\ $Spin(7)-$instantons}
We will compare the moduli of $G_{2}-$instantons to the moduli of Hermitian Yang-Mills connections, and compare the moduli of $Spin(7)-$instantons to the moduli $G_{2}-$instantons. We start by defining a Hermitian Yang-Mills connection.  \begin{Def}\label{Def HYM}Given an almost complex $6-$manifold  $X$ with a positive real $(1,1)-$form $\omega$, and a Hermitian vector bundle $E\rightarrow X$ (see Definition \ref{Def isotrivial}), a (unitary) connection $A$ is said to be Hermitian Yang-Mills if $F_{A}$ is $(1,1)$ and $\frac{\sqrt{-1}}{2\pi}F_{A}\lrcorner \omega=\mu Id_{E}$ for a real number $\mu$. The $\mu$ is called the slope of $A$. 

 When $\omega$ is co-closed i.e. $d(\omega\wedge \omega)=0$, let the degree of $E$ be defined by  $$degE\triangleq \frac{1}{2Vol_{\omega}X}\int_{X}c_{1}(E)\wedge \omega\wedge \omega,$$ where $Vol_{\omega}X=\int_{X}\frac{\omega^{3}}{3!}$. The slope of any Hermitian Yang-Mills connection on $E$ must be  $\frac{degE}{rankE}$. \end{Def}
\begin{rmk} The contraction``$\lrcorner$" between two forms, in any context, is with respect to the underlying Riemannian metric. For example, in Definition \ref{Def HYM} above, $\lrcorner=\lrcorner_{\omega}$ means the contraction with respect to (the Riemannian metric of) $\omega$. In \eqref{equ def G2 instanton} and \eqref{equ def proj G2 instanton} below,  $\lrcorner=\lrcorner_{g_{\phi}}$ means the contraction with respect to the metric of the $G_{2}-$structure $\phi$.
\end{rmk}

Next, we define $G_{2}-$instantons. 
\begin{Def}\label{Def instantons}Let $(M,\phi)$ be a $7-$manifold with a $G_{2}-$structure.  A connection $A$ on $E\rightarrow M$ is called a $G_{2}-$instanton if \begin{equation}\label{equ def G2 instanton}\star(F_{A}\wedge \psi)=0\
(\textrm{which is equivalent to}\  F_{A}\lrcorner\phi=0).\end{equation}
The connection $A$ is called a projective $G_{2}-$instanton, if there is a harmonic $\R$-valued $1-$form $\theta$ such that 
 \begin{equation}\label{equ def proj G2 instanton}\frac{\sqrt{-1}}{2\pi}\star(F_{A}\wedge \psi)=\theta Id_{E}\ (\textrm{which is equivalent to}\ \frac{\sqrt{-1}}{2\pi}(F_{A}\lrcorner \phi)=\theta Id_{E}).\end{equation}

Similarly, we define $Spin(7)-$instantons. Let $(M^{8},\Psi)$ be an $8-$manifold with a  $Spin(7)$ structure. A connection $A$ on a bundle $E\rightarrow M^{8}$ is called a $Spin(7)-$instanton if
\begin{equation}\label{equ def spin7 instanton equation}\star_{g_{\Psi}}(F_{A}\wedge \Psi)+F_{A}=0.\end{equation}
\end{Def}

 \subsection{Moduli spaces and their topology\label{section Moduli spaces and their topology}}
 We define the moduli spaces (of gauge equivalence classes) of the $3$ kinds of connections in the previous section. 
 \begin{Def}\label{Def moduli}  In view of Definition \ref{Def HYM} and \ref{Def instantons}, let \begin{equation}\mathfrak{M}_{X,E,\omega-HYM},\ \mathfrak{M}_{X,E,\omega-HYM-0},\ \mathfrak{M}_{M,E,\phi},\ \mathfrak{M}^{proj}_{M,E,\phi},\  \mathfrak{M}_{M^{8},E,\Psi},\end{equation}
 denote the set of all gauge equivalence-classes of smooth Hermitian Yang-Mills connections on $E\rightarrow X$, Hermitian Yang-Mills connections with $0-$slope on $E\rightarrow X$,
 $G_{2}-$instantons on $E\rightarrow M$, projective $G_{2}-$instantons on $E\rightarrow M$, and $Spin(7)-$instantons on $E\rightarrow M^{8}$ respectively. Let \begin{equation}\mathfrak{M}^{irred}_{X,E,\omega-HYM},\ \mathfrak{M}^{irred}_{X,E,\omega-HYM-0},\ \mathfrak{M}^{irred}_{M,E,\phi},\ \mathfrak{M}^{proj,irred}_{M,E,\phi},\  \mathfrak{M}^{irred}_{M^{8},E,\Psi}\end{equation}
denote respectively the subsets of all  irreducible (gauge equivalence-classes of) connections. 
\end{Def}
 We now recall some classical material about Hermitian Yang-Mills connections on a K\"ahler manifold and stability of a holomorphic vector bundle. 
\begin{Def}\label{Def DUY} Over a (closed) K\"ahler $3-$fold $(X_{kah},\omega)$, let $E$ be a Hermitian vector bundle. A holomorphic bundle $(E,\bar{\partial}_{\alpha})$ (on the topologic bundle $E$) is said to be slope-stable, if for any torsion free coherent sub-sheaf $\mathcal{F}$ such that $0<rank\mathcal{F}<rank E$, $\mu(\mathcal{F})<\mu(E)$. We say that $(E,\bar{\partial}_{\alpha})$ is poly-stable if it is a direct sum of stable bundles of the same slope. 

Let $\mathfrak{M}^{AG}_{X_{kah},E,[\omega]-stable}$ denote the set of all isomorphism classes of $[\omega]-$slope-stable holomorphic structures on  $E\rightarrow X_{kah}$. It is an algebro-geometric moduli.  Donaldson-Uhlenbeck-Yau Theorem (\cite{Donaldsonpaper1}, \cite{UY}, \cite{Donaldsonpaper2}) implies that for any holomorphic structure $\bar{\partial}_{\alpha}$, the following two conditions are equivalent (also see the  presentation in \cite[Theorem 8.3]{Kobayashi}). 
\begin{itemize}\item  There is a Hermitian Yang-Mills connection such that the induced holomorphic structure is isomorphic to $\bar{\partial}_{\alpha}$.
\item$(E,\bar{\partial}_{\alpha})$ is poly-stable.
\end{itemize} 
A Hermitian Yang-Mills connection induces a stable holomorphic structure if and only if it is irreducible. Moreover, the natural map  \begin{equation}\mathfrak{M}^{irred}_{X_{kah},E,\omega-HYM}\rightarrow \label{DUY}
 \mathfrak{M}^{AG}_{X_{kah},E,[\omega]-stable}
\end{equation}
is a bijection. 
\end{Def}

The moduli spaces in Definition \ref{Def moduli} can be equipped with natural topologies as follows. 
\begin{Def}(Topology of the moduli spaces)\label{Def top} Let $|\cdot|$ denote  the standard metric for  complex matrices i.e.
$$|A|^{2}=Trace(A A^{\star}),$$
 and $E\rightarrow Y$ be a Hermitian vector bundle (see Definition \ref{Def isotrivial}). Using a Riemannian metric on $Y$ (which should be clear from the context in practice),   $|\cdot|$  extends to a metric on $\Omega^{k}(adE)|_{p}$ ($\Omega^{k}(EndE)|_{p}$) for any $p\in Y$. We still denote this metric by $|\cdot|$.\begin{itemize}\item  Let $\Lambda_{E,Y}$ ($\Lambda^{irred}_{E,Y}$) denote the space of all (irreducible) gauge equivalence classes of  smooth (unitary) connections respectively. Similarly to \cite[(4.2.3)]{Donaldson}, we define a metric  on the space $\Lambda_{E,Y}$ of connections as the following. 
\begin{equation}\label{equ Def dist and norm on moduli}d_{\Lambda_{E,Y}}([A_{1}],[A_{2}])\triangleq \inf_{g\in \mathfrak{G}}|| A_{1}-g(A_{2})||.\end{equation}
The $||\cdot||$ above is a metric on $\Omega^{1}(adE) \ (\textrm{and}\ \Omega^{1}(EndE))\rightarrow Y$ defined by\\ $||\cdot||\triangleq sup_{p\in Y}|\cdot|.$ It is  invariant under the linear actions of $\mathfrak{G}$ by both  left and right multiplication in the endomorphism part. 

In general, the metric $d_{\Lambda_{E,Y}}$ induces a metric topology on any subset of $\Lambda_{E,Y}$, including those moduli spaces in Definition \ref{Def moduli} and in the main Theorem \ref{Thm Rigidity} below. 

\item Let $G$ be a compact subgroup of the gauge group $\mathfrak{G}$,  and let $CON(G)$ denote the space of all conjugacy classes of $G$. Based on the above definition, let $x,y\in CON(G)$, we consider the following  metric and the associated topology on $CON(G)$. \begin{equation} \label{equ 0 Def top} d_{CON(G)}(x,y)=\inf_{g\in G}||x-gyg^{-1}||.
\end{equation}

In practice, $G$ will usually be the stabilizer group of a connection.
\end{itemize}\end{Def}
\subsection{Main Statement}
Our  fully set up  terminology above is at our disposal to state the main result. It classifies all $G_{2}-$instantons ($Spin(7)$-instantons) on the product manifolds, and  confirms the existence of a ``dimension reduction" for their moduli spaces. In section $\mathbb{I}$ of the main result below,  we state  $1:$ the equivalent condition for the existence of a $G_{2}-$instanton; $2:$ ``broken gauge" equivalence of a $G_{2}-$instanton on the trivial circle bundle and a Hermitian Yang-Mills connection on the base;  $3$ and $4$: a  ``fibration" structure of the moduli spaces. \textit{The statements in $\mathbb{II}$ for $Spin(7)-$instantons and $\mathbb{III}$ for projective $G_{2}-$instantons parallel those in $\mathbb{I}$}. 
\begin{thm}\label{Thm Rigidity}$\mathbb{I}$:  Given a  $6-$dimensional manifold  $X$ with a half-closed $SU(3)-$structure  $(J,g_{X},\omega,\Omega)$ and a Hermitian vector bundle $E\rightarrow X$, on the pullback bundle\\ $\pi^{\star}(E)\rightarrow X\times S^{1}$ and with respect to the product $G_{2}-$structure  \eqref{equ G2 structure on the product} on $X\times S^{1}$,  the  following is true. 
\begin{enumerate}\item $\pi^{\star}E\rightarrow X\times S^{1}$ admits a $G_{2}-$instanton if and only if $E\rightarrow X$ admits a Hermitian Yang-Mills connection with $0-$slope.

 Consequently, when $(X,J,g_{X},\omega,\Omega)$ is Calabi-Yau, $\pi^{\star}E\rightarrow X\times S^{1}$ admits a $G_{2}-$instanton if and only if $E\rightarrow X$ admits a poly-stable holomorphic structure and $degE=0$. 
\item A connection  on $\pi^{\star}E\rightarrow X\times S^{1}$ is a $G_{2}-$instanton if and only if it is iso-trivial with respect to a Hermitian Yang-Mills connection with $0-$slope on $E\rightarrow X$. 
\item $\mathfrak{M}_{X\times S^{1}, \pi^{\star}E, \phi}$, if non-empty,  admits a continuous surjective map $\rho$ to  $\mathfrak{M}_{X,E,\omega-HYM-0}$. For any $[B]\in \mathfrak{M}_{X,E,\omega-HYM-0}$, $\rho^{-1}([B])$ is homeomorphic  to $CON(\Gamma_{B})$. 
\item $\mathfrak{M}^{irred}_{ X\times S^{1},\pi^{\star}E,\phi}=\rho^{-1}(\mathfrak{M}^{irred}_{X,E,\omega-HYM-0})$, and both of them are homeomorphic to\\ $S^{1}\times \mathfrak{M}^{irred}_{X,E,\omega-HYM-0}$. 

 Consequently, when $(X,J,g_{X},\omega,\Omega)$ is Calabi-Yau and $degE=0$, $\mathfrak{M}^{irred}_{X\times S^{1}, \pi^{\star}E,\phi}$  is bijective to $S^{1}\times \mathfrak{M}^{AG}_{X,E,[\omega]-stable}$. 
\end{enumerate}

$\mathbb{II}$: Given a $7-$dimensional manifold $M$ with a co-closed $G_{2}-$structure $\phi$,  and a Hermitian vector bundle $E\rightarrow M$, on the pullback bundle $\pi^{\star}E\rightarrow M\times S^{1}$ and with respect to the product $Spin(7)-$structure on $M\times S^{1}$ in \eqref{equ Spin7 structure on the product}, the following is true.  
\begin{enumerate}\item $\pi^{\star}E\rightarrow M\times S^{1}$ admits a $Spin(7)-$instanton if and only if $E\rightarrow M$ admits a $G_{2}-$instanton. 
\item A connection on $\pi^{\star}E\rightarrow M \times S^{1}$ is a   $Spin(7)-$instanton if and only if it is iso-trivial with respect to a $G_{2}-$instanton on  $E\rightarrow M$. 
\item $\mathfrak{M}_{M\times S^{1},\pi^{\star}E,\Psi}$, if non-empty, admits a continuous surjective map $\rho$ to $\mathfrak{M}_{M,E,\phi}$. For any $[B]\in \mathfrak{M}_{M,E,\phi}$, $\rho^{-1}([B])$ is homeomorphic  to $CON(\Gamma_{B})$.
 \item $\mathfrak{M}^{irred}_{ M\times S^{1},\pi^{\star}E,\Psi}=\rho^{-1}(\mathfrak{M}^{irred}_{M,E,\phi})$, and both of them are homeomorphic to $S^{1}\times \mathfrak{M}^{irred}_{M,E,\phi}$. 
\end{enumerate}

$\mathbb{III}$ (projective version of $\mathbb{I}$): Under the same conditions and setting in  $\mathbb{I}$, we assume additionally that $H^{1}(X,\R)=0$.  Then the following is true.
\begin{enumerate}\item $\pi^{\star}E\rightarrow X\times S^{1}$ admits a projective $G_{2}-$instanton if and only if $E\rightarrow X$ admits a Hermitian Yang-Mills connection. 

Consequently, when $(X,J,g_{X},\omega,\Omega)$ is Calabi-Yau, $\pi^{\star}E\rightarrow X\times S^{1}$ admits a projective $G_{2}-$instanton if and only if $E\rightarrow X$ admits a poly-stable holomorphic structure. 
\item A connection  on $\pi^{\star}E\rightarrow X\times S^{1}$ is a projective $G_{2}-$instanton if and only if it is iso-trivial with respect to a Hermitian Yang-Mills connection  on $E\rightarrow X$.

\item $\mathfrak{M}^{proj}_{X\times S^{1},\pi^{\star}E,\phi}$, if non-empty,  admits a continuous surjective map $\rho$ to  $\mathfrak{M}_{X,E,\omega-HYM}$. For any $[B]\in \mathfrak{M}_{X,E,\omega-HYM}$, $\rho^{-1}([B])$ is homeomorphic  to $CON(\Gamma_{B})$. 
\item $\mathfrak{M}^{proj,\ irred}_{X\times S^{1},\pi^{\star}E,\phi}=\rho^{-1}(\mathfrak{M}^{irred}_{X,E,\omega-HYM})$, and both of them are homeomorphic to\\ $S^{1}\times \mathfrak{M}^{irred}_{X,E,\omega-HYM}$. Consequently, when $(X,J,g_{X},\omega,\Omega)$ is Calabi-Yau, $\mathfrak{M}^{proj,irred}_{X\times S^{1},\pi^{\star}E,\phi}$ is bijective to $S^{1}\times \mathfrak{M}^{AG}_{X,E,[\omega]-stable}$. 
\end{enumerate}
\end{thm}

\begin{rmk}\label{rmk non pull back instantons} The pullback of any  Hermitian Yang-Mills connection $B$ with $0-$slope on\\ $E\rightarrow X$ to $\pi^{\star}E\rightarrow X\times S^{1}$ is a $G_{2}-$instanton (see \cite[Example 1.93]{Walpuski} for example). Proposition \ref{Prop gauge equivalence characterization}  and Lemma \ref{lem there exists a gauge connecting id to arbitrary a} below imply that there exist $G_{2}-$instantons on the product manifold which is not  gauge equivalent to any such pullback. Nevertheless, by Theorem \ref{Thm Rigidity}.$\mathbb{I}.2$, any such instanton must be iso-trivial even if it is not a pullback. 
\end{rmk}
\begin{rmk} Investigations by Walpuski, S\'a Earp, Nordstr\"om, Menet etc show that the moduli of  $G_{2}-$instantons on certain closed $7-$manifolds are non-empty (see \cite{Walpuski}, \cite{Menet} and the references therein). The point of this note is the full moduli. 
\end{rmk}
\begin{rmk}When the $G_{2}-$structure $\phi$ on $X\times S^{1}$ is not co-closed, it seems  natural to work with  $G_{2}-$monopoles rather than instantons (see \cite[(25) and the enclosing page]{DonaldsonSegal}). However,  the proof of Theorem \ref{Thm Rigidity}.$\mathbb{I}$ indicates that it is reasonable to  work with  instantons. 
\end{rmk}
Schematically, we can understand Theorem \ref{Thm Rigidity}.$\mathbb{I}.4$ as follows:  \textit{if the $7-$manifold is a trivial circle bundle over a certain $6-$manifold} satisfying certain conditions, then under the special data above, \textit{the moduli of irreducible $G_{2}-$instantons on the $7-$manifold is also a trivial circle bundle over the moduli of irreducible Hermitian Yang-Mills connections with $0-$slope on the $6-$manifold}. 
\subsection{Ideas of the proof}
We sketch of the proof of Theorem \ref{Thm Rigidity}.$\mathbb{I}$ as follows. The proof of $\mathbb{II}$ and $\mathbb{III}$ is similar. 

Step 1: Similarly to the $3-$dimensional case, modulo gauge,  a   $G_{2}-$instanton on $X\times S^{1}$ can be understood as a``periodic" orbit of the gradient flow  of the Chern-Simons functional on $X$ [see \eqref{equ condition 1 G2 instant on product} and \eqref{equ condition 2 G2 instant on product}]. The point is that, although the Chern-Simons functional is not necessarily  gauge invariant, it is invariant along any smooth one-parameter gauge orbit (Lemma \ref{lem CS constant along gauge orbit}). Applying a smooth one-parameter family of gauges (initiated from $Id_{E}$) to the instanton equation  \eqref{equ condition 1 G2 instant on product},  the monotonicity implies  that the curvature term in  \eqref{equ condition 1 G2 instant on product} vanishes. Then Lemma \ref{lem ddt of gauge action on a connection} below allows  us to ``integrate" the instanton equation, which shows that the instanton is iso-trivial with respect to a connection $B$ on the base manifold.

Step 2: To establish the bijection from $CON(\Gamma_{B})$ to $G_{2}-$instantons iso-trivial with respect to $B$ (Lemma \ref{lem easy continuity}), we need the existence (Lemma \ref{lem there exists a gauge connecting id to arbitrary a}) saying that any element in $\Gamma_{B}$ can be connected to $Id_{E}$ via a $B-$admissible broken gauge. The structure group (of the bundle) being $U(m)$ is crucial for this purpose. The argument does not generalize obviously to $SU(m)$ or $SO(m)$. 

Step 3: The properties of the natural topology in Definition \ref{Def top} yield the continuity and homeomorphism properties of the maps characterizing the moduli space (see Proposition \ref{prop hard continuity} below). 

 We hope that the following additional diagram might be helpful. 
 \begin{center}
\begin{tikzpicture}[->,>=stealth',shorten >=1pt,auto,node distance=2.5cm,
  thick,main node/.style={rectangle,draw}]

   \node[draw,align=left] at (0,0) (1)  {Theorem \ref{Thm Rigidity} $\mathbb{I}1-2$,\\  $\mathbb{II}1-2, \mathbb{III}1-2$} ;
     \node[draw,align=left] (2) at (0,-1) {Theorem \ref{Thm Rigidity} $\mathbb{I}3-4$,\\  $\mathbb{II}3-4, \mathbb{III}3-4$} ;
      \node[draw,align=left] (3) [above of=1]  {Prop \ref{Prop gauge equivalence characterization}\\ (criterion for\\ gauge equivalence)} ;
     \node[draw,align=left] (4) at (-4,0.5) {Lem \ref{lem ddt of gauge action on a connection}\\ ``integration" of\\ the endorphisms)} ;
      \node[draw,align=left] (5) at (5,0) {Lem \ref{lem CS constant along gauge orbit}\\ (invariance of the\\ Chern-Simons functional\\ along smooth gauge orbits)} ;
      \node[draw,align=left] (6) at (-3,2.5) {Lem \ref{lem there exists a gauge connecting id to arbitrary a}\\ (existence of \\ $B-$admissible\\ broken gauge)} ;  \node[draw,align=left] at (7,2.7) {The\\ arrows\\ mean\\ implying.};
        \node[draw,align=left] (7) at (3.5,2.5) {Lem \ref{first variation of CS} \\ (first variation\\ and monotonicity \\ of the Chern-Simons\\ function)} ;
            \node[draw,align=left] (8) at (-4,-1) {Lem \ref{lem irred}\\ (irreducibility)} ;  
                  \node[draw,align=left] (9) at (0,-2.5)  {Prop \ref{prop hard continuity}} ;
                \node[draw,align=left] (10) at (2.5,-2.5)  {Cor \ref{cor rho continuous}} ;
                       \node[draw,align=left] (11) at (5,-1.7)  {Lem \ref{Lem top up and down}} ;
                        \node[draw,align=left] (12) at (-3,-2.5)  {Lem \ref{lem easy continuity}} ;
       \path[every node/.style={font=\sffamily\small}]
    (3) edge node [right] {} (1)
    (4) edge node [right] {} (1)
    (5) edge node [right] {} (1)
      (6) edge node [right] {} (1)
    (7) edge node [right] {} (1)
    (8) edge node [right] {} (2)
      (9) edge node [right] {} (2)
    (10) edge node [right] {} (2)
    (11) edge node [right] {} (10)
     (12) edge node [right] {} (9);
 \end{tikzpicture}
\end{center}

Results on $S^{1}-$invariant $G_{2}-$instantons on  Calabi-Yau links are obtained by Calvo-Andrade -Rodr\'iguez D\'iaz-S\'a Earp \cite{Earp}.
\subsection{Simple examples}

We now  attempt to find new examples. Except for trivial bundles on Calabi-Yau manifolds$\times S^{1}$, on which all instantons with respect to the product $G_{2}-$structure are flat, it is hard to determine the topological type of a  moduli of $G_{2}-$instantons.  Nevertheless,  we do obtain the topological type of the moduli of projective $G_{2}-$instantons on a certain non-trivial bundle.  \begin{cor}\label{cor example Thomas} There exist a  smooth anti-canonical  hyper-surface $X_{CY}$ in\\ $CP^{1}\times CP^{1}\times CP^{2}$, a K\"ahler-metric $\omega$ on $X_{CY}$, a nowhere-vanishing holomorphic $(3,0)-$form $\Omega$ on $X_{CY}$, and a rank $2$ Hermitian vector bundle $E\rightarrow X_{CY}$ with the following property. Let $\phi$  be as  \eqref{equ G2 structure on the product}, then $\mathfrak{M}^{proj}_{X_{CY}\times S^{1},\pi^{\star}E, \phi}$ and $\mathfrak{M}^{proj,irred}_{X_{CY}\times S^{1},\pi^{\star}E, \phi}$ are both  homeomorphic to $S^{1}$.
\end{cor}

 The above example might only be a drop in those which could  be produced by Theorem \ref{Thm Rigidity}. For instance, by understanding the full moduli of stable structures on Jardim's instanton bundles \cite{Jardim}, we can hope to determine topological types of moduli spaces  of  $G_{2}-$instantons on certain non-trivial bundles. Similar methods apply  on nearly-K\"ahler manifolds. For example, we can start from understanding the full moduli of  the canonical connection on the tangent bundle of $S^{6}$ (see \cite{Harland}). 

This note is organized as follows. Most of the definitions are in the introduction. In section \ref{section isotrivial}, we discuss the fundamental properties of iso-trivial connections. These hold  generally  and do not involve the instanton or Hermitian Yang-Mills condition.  We prove Theorem \ref{Thm Rigidity} and Corollary \ref{cor example Thomas} in section \ref{proof} and \ref{top}.  In the Appendix, we collect some technical ingredients which are more routine than those in the main body. \\

\textbf{Acknowledgement:} The author is grateful to Simon Donaldson for helpful discussions. This work is supported by Simons Collaboration on Special Holonomy in Geometry, Analysis, and Physics. The author thanks the anonymous referee for his/her suggestions. 
\section{Preliminary on iso-trivial connections \label{section isotrivial}}
Without involving the instanton or Hermitian Yang-Mills condition,   we  establish a theory for the iso-trivial connections and admissible broken gauges alone. 
\subsection*{Elementary facts related to the broken gauges} 
Let $u$, $\chi$, $A_{0}$ be a $t-$family of  automorphisms, endomorphisms, connections of $E\rightarrow Y$ respectively which are continuously differentiable in $t\in I$, $I$ is an open interval in $\mathbb{R}$.   Suppose $\frac{\partial u}{\partial t}=u\chi$, routine calculation shows that 
\begin{equation}\label{equ ddt of gauge action on a connection}\frac{\partial u(A_{0})}{\partial t}=d_{u(A_{0})}\chi+u^{-1}(\frac{\partial A_{0}}{\partial t})u,
\end{equation}
where we used the identity 
\begin{equation}\label{equ conjugate derivative} u^{-1}(d_{A_{0}}\chi) u= d_{u(A_{0})}(u^{-1}\chi u).
\end{equation}

We need the following classical existence and uniqueness for ordinary differential equations of endomorphisms. 
\begin{lem}\label{lem ODE}  Let $E\rightarrow Y$ be a Hermitian vector bundle (see Definition \ref{Def isotrivial}). Let $\chi_{i}$ ($i=1,2$) be smooth sections to  $\pi^{\star}EndE \rightarrow Y\times(a-\epsilon,b)$, $-\infty<a<b<+\infty$, $\infty>\epsilon>0$. Then for any smooth section $s_{0}$ of $EndE\rightarrow Y$, the initial value problem 
\begin{equation}\frac{ds}{dt}=\chi_{1} s+s\chi_{2},\ s(a)=s_{0}
\end{equation}
admits an unique smooth solution $s$  on $Y\times (a-\epsilon,b)$. Moreover, when $\chi_{i}$ are all $adE-$valued and $s_{0}$ is a (unitary) gauge on $Y$,  $s$ is a gauge on $Y\times(a-\epsilon,b)$. \end{lem}

For the reader's interest, Lemma \ref{lem ODE} can be proved by  
the existence, uniqueness (see \cite[Theorem 3.1]{FB}),  and  Gronwall-inequality (see \cite[Page 12]{Burke}). 

We now turn to the criteria for the smoothly periodic extension of a smooth endomorphism of the pullback bundle on $Y\times [0,2\pi]$.
 \begin{clm}\label{clm periodic gauges}Under the setting of Definitions \ref{Def smooth periodicity} and \ref{Def Cinfty top},  suppose $s$ is a smooth section of $\pi^{\star}EndE\rightarrow Y\times [0,2\pi]$. Then $s$ extends to a  smooth (periodic) section on $\pi^{\star}End E\rightarrow Y\times S^{1}$ if and only if $\frac{\partial^{k}s}{\partial t^{k}}(0)=\frac{ \partial^{k}s}{\partial t^{k}}(2\pi)$ for any $k\geq 0$. 
 \end{clm}
The proof of Claim \ref{clm periodic gauges}, in view of Definition \ref{Def smooth periodicity}, is a routine (but interesting) exercise on multi-variable calculus.  We note that the ``only if" in Claim \ref{clm periodic gauges} is obvious. The point is to show the ``if" by the patching condition. 
 \begin{rmk}\label{rmk Y} When the underlying manifold is $Y\times S^{1}$, we add $Y$ as a subscript if the operation (gauge transformation, derivative etc) is on $Y$. For example, see \eqref{equ splitting of connection}, \eqref{equ splitting of curvature}, and \eqref{equ Spin7 instanton on the product} below. Hence in the setting of Definition \ref{Def ab gauges and iso trivial} (iso-trivial connections), for any gauge $u$ on $Y\times [0,2\pi]$, we have the following splitting on $Y\times [0,2\pi]$. 
 \begin{equation}\label{equ gauge action on product vs slice}u(B)=u_{Y}(B)+\chi_{u}dt. 
 \end{equation}
 \end{rmk}
 \subsection*{Gauge equivalence of iso-trivial connections}The following proposition determines whether two iso-trivial connections are gauge equivalent. The proof utilizes the above facts on endomorphisms. 
\begin{prop}\label{Prop gauge equivalence characterization}In the setting of Definition \ref{Def ab gauges and iso trivial}, on the product manifold $Y\times S^{1}$, two iso-trivial connections $u(B)$ and $v(\widetilde{B})$ are gauge equivalent if and only if there is a gauge $g$ on $Y$ with the following properties. 
\begin{enumerate}\item $g(B)=\widetilde{B}$,\
\item  $u(2\pi)g=gv(2\pi)$.
 \end{enumerate}
 Under the above two conditions, the gauge that transforms $u(B)$ to $v(\widetilde{B})$ is $s=u^{-1}gv$ i.e. $s[u(B)]=v(\widetilde{B})$.
\end{prop}
\begin{proof} We first show the ``only if''. On $Y\times (0,2\pi)$, $(us)(B)=v(\widetilde{B})$ means $g(B)=\widetilde{B}$ where $g \triangleq usv^{-1}$. Then the identity \eqref{equ gauge action on product vs slice} yields that $\frac{\partial g}{\partial t}=0$ i.e. $g$ is independent of  $t\in (0,2\pi)$.

Let $t\rightarrow 0$ in $g \triangleq usv^{-1}$, we find that  $g=s(0)$. Because $s(2\pi)=s(0)=g$, let  $t\rightarrow 2\pi$, we find that  $u(2\pi)g=gv(2\pi)$. 
 
 The proof of the ``if'' is simply by taking $s=u^{-1}gv$. We compute $\frac{\partial s}{\partial t}$:
 \begin{equation*}\frac{\partial s}{\partial t}=-\chi_{u}s+s\chi_{v}.
 \end{equation*}  Because both $\chi_{u}$ and $\chi_{v}$ are smoothly periodic, so is $\frac{\partial s}{\partial t}$.  By the periodicity condition  $s(2\pi)=s(0)$ and Claim  \ref{clm periodic gauges}, $s$ is smoothly periodic.
 \end{proof}
 

 \subsection*{Connecting the identity automorphism to an arbitrary element in $\Gamma_{B}$ via a $B-$admissible broken gauge}
 We  show that for any connection $B$ on $Y$, any element in the stabilizer group is the value of a $B-$admissible broken gauge at $t=2\pi$. This is crucial to showing that each ``fiber" is bijective via $\rho$ to $CON (\Gamma_{B})$, in relation to Theorem \ref{Thm Rigidity}.$\mathbb{I},\mathbb{II},\mathbb{III}.3$.
 \begin{lem} \label{lem there exists a gauge connecting id to arbitrary a}Still in the setting of Definition \ref{Def ab gauges and iso trivial},  for any connection $B$ on $E\rightarrow Y$, and any $a\in \Gamma_{B}$, there is a $B-$admissible broken  gauge $u$ on the pullback $\pi^{\star} {E}\rightarrow Y\times
[0,2\pi]$ such that $u(2\pi)=a$. 
 \end{lem}
 \begin{proof} Step 1: For any $a\in \Gamma_{B}$, we first show that there is an automorphism  $\tau$ on $Y\times S^{1}$ such that $\tau(2\pi)=a$ and $\tau$ satisfies all requirements for  $B-$admissibility except being unitary.  \begin{clm} \label{clm there exists a good path}There exists a smooth curve $\gamma(t):\ [0,2\pi]\rightarrow \C$ such that the following holds
 \begin{itemize} 
 \item $\gamma(t)=1$ when $t\in [0,\frac{1}{10}]$.  $\gamma(t)=0$ when $t\in [-\frac{1}{10}+2\pi,2\pi]$. 
 \item $\tau\triangleq a+\gamma(t)(Id-a)$ is a section of $Aut(E)$ i.e. it is invertible for every $t\in [0,2\pi]$. 
 \end{itemize}
 \end{clm}
 
To prove Claim \ref{clm there exists a good path}, we note that at any $p\in Y$, $det[a+x(Id-a)]$ is a degree $m$ polynomial in $x$. As a section of $End(E)\rightarrow Y$, we find that 
\begin{equation}\label{equ 1 lem there exists a gauge connecting id to arbitrary a}d_{B} [a+x(Id-a)]=0.
\end{equation}
To show that $a+x(Id-a)$ is always invertible except for finitely-many $x$, we need the following. 
 \begin{clm}\label{clm det is a constant} $H\in C^{\infty}[Y,EndE]\ \textrm{and}\ d_{B}H=0\Longrightarrow\ det(H)\ \textrm{is a constant on}\ Y$. \end{clm}
To prove the claim, it suffices to show $det(H)$ is a constant on any smooth curve $l(t)$,\\ $t\in [0,t_{0}]$ connecting  two arbitrary  distinct points $p, q\in Y$. Parallel transport yields a $B-$parallel frame $S(t)=[s_{1}(t),...,s_{i}(t),...,s_{m}(t) ]$  along $l(t)$. For any tangent vector $X$ at a point $p$, let $\nabla_{B,X}$ denote the derivative at $p$ along $X$ with respect to the connection $B$. Let $h$ be the matrix of $H$ under $S(t)$ i.e. $HS=Sh$ on $l(t)$, then $d_{B}H=0$ implies that
$$0=\nabla_{B,\dot{l}(t)}HS= \nabla_{B,\dot{l}(t)}Sh=S\frac{\partial h}{\partial t}.$$ The above  means that the matrix $h$ is independent of $t$. Using that $det(H)=det (h)$ on $l(t)$, and that at any point,  $det(H)$ is independent of frame, the proof of Claim \ref{clm det is a constant} is complete. 

Applying Claim \ref{clm det is a constant} and  condition  \eqref{equ 1 lem there exists a gauge connecting id to arbitrary a} to $H=\tau\triangleq a+\gamma(t)(Id-a)$, the roots $x_{i}, i=1...m$ of the polynomial $det[a+x(Id-a)]=0$ (counted with multiplicities) must be constants on $Y$. The topological space $\C \setminus \cup_{i=1}^{m}x_{i}$ is path connected. Because $det[a+x(Id-a)]\neq 0$ when $x=1$ or $x=0$, there is a $\gamma(t)$ which not only  satisfies the first  desired condition in Claim  \ref{clm there exists a good path}, but also avoids the roots $\cup_{i=1}^{m}x_{i}$. Then the second  desired condition in Claim  \ref{clm there exists a good path} holds.

Step 2: we then improve $\tau$ to be  unitary. The following key ingredient holds by elementary proof.  Let $Herm_{m\times m}$ ($Herm_{m\times m}^{+}$) denote the set of all $m\times m$ (positive definite) Hermitian matrices.  
\begin{clm}\label{clm unique square root}For any $H\in Herm_{m\times m}^{+}$, there exists a unique $h\in Herm_{m\times m}^{+}$ such that $H=h^{2}$. We  denote $h$ by $\sqrt{H}$. 
\end{clm}

 Let $N\in GL(m,\C)$ be an invertible complex matrix. Using the square root above, we  define the linear operator ``projecting" an invertible matrix to a unitary one.
 \begin{equation}P(N)=(\sqrt{NN^{*}})\cdot N^{*,-1}.
 \end{equation}
It is routine to verify that 
\begin{eqnarray}& & P(N)\in U(m)\ \textrm{for any}\ N\in GL(m,\C).\ P(N)=N\ \textrm{if}\ N\in U(m).\label{eqn P preserves unitary}
\\& & P(g^{-1}N g)=g^{-1}P(N)g\ \textrm{if}\ g\in U(m).\label{eqn P commutes with conjugation}
\end{eqnarray}
Let $\tau$ be an automorphism on $Y\times [0,2\pi]$. On each coordinate chart $U\times [0,2\pi]$ of the pullback bundle on $Y$, under the pullback trivialization $\pi^{\star}s_{U}$, still let $\tau$ denote the matrix-valued function representing the automorphism $\tau$.   The transition condition \eqref{eqn P commutes with conjugation} says that the automorphism $u$ defined by $u(\pi^{\star}s_{U})\triangleq (\pi^{\star}s_{U}) P(\tau)$ is independent of the coordinate or trivialization chosen. Thus $u$ is a global unitary automorphism.    Moreover, 
\begin{itemize}
 \item $P$ is analytic in $N\in GL(m,\C)$ (see Lemma \ref{lem sqrt is real analytic} below).  Then  $u$ is smooth since $\tau$ is. 
\item 
$\chi_{\tau}\triangleq \tau^{-1}\frac{\partial \tau}{\partial t}=0$ when $t$ is close to $0$ or $2\pi$. By the fact \eqref{eqn P preserves unitary}, $u=\tau$ there. Then  $\chi_{u}=0$ when $t$ is close to $0$ or $2\pi$. Claim \ref{clm periodic gauges} thereupon says that $\chi_{u}$ is smoothly periodic. 
\end{itemize}
The above  precisely means that $u$ is $B-$admissible (see Definition \ref{Def ab gauges and iso trivial}). The proof of Lemma \ref{lem there exists a gauge connecting id to arbitrary a} is complete. 
 \end{proof}
 
 \subsection*{Irreducibility}
 In the following, we show that an iso-trivial connection with respect to  $B$ is reducible if and only if the connection $B$ is reducible. Hence, the same statement holds if we replace ``reducible" by ``irreducible". For Theorem \ref{Thm Rigidity}.$\mathbb{I},\mathbb{III}.4$, this is crucial in showing that the moduli of irreducible instantons on the $7-$manifold maps to the moduli of irreducible Hermitian Yang-Mills connections on the $6-$manifold. The same applies to  $\mathbb{II}.4$ therein as well.

 \begin{lem}\label{lem irred}  Given an isotrivial connection $u(B)$ on $E\rightarrow Y\times S^{1}$,  for any gauge $v$ on $Y\times S^{1}$, the following two conditions are equivalent. 
 \begin{enumerate}\item $d_{u(B)}v=0$.
\item There is an element $b\in \Gamma_{B}$ such that $bu(2\pi)=u(2\pi)b$ and $v=u^{-1}bu$. \end{enumerate}
Consequently, $u(B)$ is reducible on $Y\times S^{1}$ $\Longleftrightarrow$ $B$ is reducible on $Y$.
 \end{lem}
  \begin{proof} Routine computation shows 
\begin{equation} d_{u(B)}v=d_{Y,u(B)}v+(\frac{\partial v}{\partial t}+[\chi_{u},v])dt.
\end{equation}
\begin{equation}\textrm{Then}\ \label{equ 0 lem irred}d_{u(B)}v=0\Longleftrightarrow \left\{ \begin{array}{c} \frac{\partial v}{\partial t}+[\chi_{u},v]=0,\\ d_{Y,u(B)}v=0. \end{array}\right.
\end{equation}
The first identity on the right  implies 
\begin{equation}\label{equ 0.5 lem irred}\frac{\partial (uvu^{-1})}{\partial t}=0.\end{equation}
$\textrm{Let}\ b\triangleq v(0)$, assuming ``1" and using the second identity on the right hand side of \eqref{equ 0 lem irred}, we have $b\in \Gamma_{B}$. The vanishing \eqref{equ 0.5 lem irred} shows 
\begin{equation}\label{equ 0 lem parallel gauge}v=u^{-1}bu\ \textrm{for all}\ (p,t)\in X\times S^{1}.
\end{equation}

``$1\Longrightarrow 2$'' : Because $v(0)=v(2\pi)=b$, it follows from evaluating \eqref{equ 0 lem parallel gauge} at $t=2\pi$. 

``$2\Longrightarrow 1$'' : The conditions in $``2"$ imply that $v(0)=v(2\pi)=b$. This means that $v$ is periodic. Because $\chi_{u}$ is smoothly periodic, successively differentiating $v=u^{-1}bu$ in $t$  shows that for any $k\geq 1$, $\frac{\partial^{k}v}{\partial t^{k}}$ is also periodic. Thus $v$ satisfies the conditions in Claim \ref{clm periodic gauges}, which thereupon says that $v$ is smoothly periodic. Applying the identity  \eqref{equ conjugate derivative} to the easy equation $u^{-1}(d_{Y,B}b)u=0$, we  verify the 2  conditions on the right hand side of \eqref{equ 0 lem irred} which are equivalent to $d_{u(B)}v=0$. This means $``1"$ holds. 

For the last conclusion in Lemma \ref{lem irred}, we first prove ``$\Longrightarrow$''. Suppose $u(B)$ is reducible, then there is a point $(p,t)\in X\times S^{1}$ and a $v$  such that $d_{u(B)}v=0$ but  $v(p,t)\notin Center[U(m)]$. 
By $``2"$,  $b\notin Center[U(m)]$: if not, $v=b\in Center[U(m)]$ at $(p,t)$. This is a contradiction.  Then $B$ is reducible.

We then prove ``$\Longleftarrow$''. Suppose $B$ is reducible. 

 If $u(2\pi)\in Center[U(m)]$, let $b$ be an arbitrary element in the non-empty set\\   $\Gamma_{B}\setminus Center[U(m)]$.
Then $bu(2\pi)=u(2\pi)b$.  The implication ``$2\Longrightarrow 1$'' says that $v\triangleq u^{-1}b u$ satisfies $d_{u(B)}v=0$ and $v(0)\notin Center[U(m)]$. Hence $u(B)$ is reducible on $Y\times S^{1}$. 
 
  If $u(2\pi)\notin Center[U(m)]$,  let $b=u(2\pi)\in \Gamma_{B}\setminus Center[U(m)]$, then $bu(2\pi)=u(2\pi)b$ still holds. Let $v\triangleq u^{-1}bu$, we still get $v(0)\notin Center[U(m)]$ and $d_{u(B)}v=0$. Hence $u(B)$ is reducible on $Y\times S^{1}$. \end{proof}
\section{Chern-Simons functionals and  proof of Theorem \ref{Thm Rigidity} I1, I2, II1, II2, III1, III2 \label{proof}}
\subsection{Chern-Simons functional on an arbitrary closed manifold}
To prove iso-triviality in the main theorem, and to deal with the instanton equations (for example, see \eqref{equ condition 1 G2 instant on product} below), we need a version of the Chern-Simons functional. The monotonicity and invariance along a smooth gauge orbit of the functional will play a crucial role. 
   
\begin{Def}\label{Def CS}Let $E\rightarrow Y$ be a  Hermitian vector bundle (see Definition \ref{Def isotrivial}). Given a closed $(n-3)-$form $H$ on $Y$, and a smooth (reference) connection $A_{0}$ on $E$, let the independent variable $a$ be an $adE-$valued $1-$form on $Y$. We define the Chern-Simons functional $CS_{Y,H}$ as follows. 
\begin{equation}\label{equ def CS functional}CS_{Y,H}(a)=\int_{Y}Tr(a\wedge d_{A_{0}}a+\frac{2}{3}a\wedge a\wedge a+2a\wedge F_{A_{0}})\wedge H.
\end{equation}

\end{Def}
In conjunction with the convention in Remark \ref{rmk Y}, any smooth connection $A$ on the pullback $\pi^{\star}E\rightarrow Y\times S^{1}$ can be written as 
\begin{equation} \label{equ splitting of connection} A=A_{Y}+\chi dt,
\end{equation}
where $A_{Y}=A_{Y}(t)$ is a smooth connection on $\pi^{\star}E\rightarrow Y\times S^{1}$ without $dt-$component, and $\chi$ is a smooth section of $\pi^{\star}(adE)\rightarrow  Y\times S^{1}$. The $dt-$component $\chi dt$ is well defined globally because the bundle is a pullback from $Y$.  The transition function is independent of $t$,  thus a local $dt-$component does not depend on the coordinate neighborhood chosen. Resultantly,  the difference $A_{Y}=A-\chi dt$ is also a globally well defined connection. In particular, both $A_{Y}$ and $\chi$ are smoothly periodic.  

 In view of the splitting of connection in \eqref{equ splitting of connection}, the curvature of $A$ on $Y\times S^{1}$ splits as
\begin{equation} \label{equ splitting of curvature}F_{A}=F_{Y, A_{Y}}+(d_{Y,A_{Y}} \chi-\frac{\partial A_{Y}}{\partial t})\wedge dt.\end{equation}

In order to produce an admissible broken gauge from an instanton, we need the following. 
\begin{lem} \label{lem ddt of gauge action on a connection}  In the setting of  Definitions \ref{Def isotrivial} and \ref{Def CS}, let $E\rightarrow Y$ be a Hermitian vector bundle and  suppose $A_{Y}$ is  a smooth  connection on $\pi^{\star}E\rightarrow Y\times [0,2\pi]$ without $dt-$component.

 $\mathbb{I}:$ Suppose $\frac{\partial A_{Y}}{\partial t}=b+d_{A_{Y}}\chi$ for  two arbitrary smooth sections $b$ and $\chi$ to $\pi^{\star}EndE\rightarrow Y\times S^{1}$. Let $s$ be the solution to the following equation produced by Lemma \ref{lem ODE}. 
 \begin{equation}\label{equ0 Lemma ddt of gauge action on a connection}\frac{\partial s}{\partial t}=-\chi s,\ s(0)=Id,\ t\in [0,2\pi).\end{equation} Then  
 \begin{equation}\label{equ Lemma ddt of gauge action on a connection}\frac{\partial s_{Y}(A_{Y})}{\partial t}=s^{-1}bs.
 \end{equation}

 $\mathbb{II}:$ Suppose further that $A_{Y}$ is smoothly periodic. The following conditions are equivalent.
 \begin{enumerate}
\item $\frac{\partial A_{Y}}{\partial t}=d_{A_{Y}}\chi$ for a  section $\chi$ of $\pi^{\star}adE\rightarrow Y\times S^{1}$.
\item There exists a smooth gauge $u$ on $\pi^{\star}E\rightarrow Y\times [0,2\pi]$ such that $A_{Y}=u_{Y}[A_{Y}(0)]$, $u(0)=Id$,  $u(2\pi)\in \Gamma_{A_{Y}(0)}$, and $u^{-1}\frac{\partial u}{\partial t}$ is smoothly periodic. 
 \end{enumerate}
 Moreover, the correspondence is given by $\chi=u^{-1}\frac{\partial u}{\partial t}$. 
 \end{lem}
 \begin{proof} Via routine calculation, $\mathbb{I}.\eqref{equ Lemma ddt of gauge action on a connection}$ is a direct corollary of the identity \eqref{equ ddt of gauge action on a connection} on the derivative in $t$. For $\mathbb{II}$, we first  show that $1\Longrightarrow 2$.  Let $b=0$ in \eqref{equ Lemma ddt of gauge action on a connection}, we find \begin{equation}\frac{\partial s_{Y}(A_{Y})}{\partial t}=0\ i.e. \ s_{Y}(A_{Y}) \ \textrm{is independent of}\ t. 
 \end{equation} Then $A_{Y}=s_{Y}^{-1}[A_{Y}(0)]$. Let $u\triangleq s^{-1}$, by \eqref{equ0 Lemma ddt of gauge action on a connection}, $u(0)=Id$.  Because $A_{Y}$ is smoothly periodic, we have that $u(2\pi)\in \Gamma_{A_{Y}(0)}$. Moreover, we  compute via \eqref{equ0 Lemma ddt of gauge action on a connection} that
 \begin{equation}\label{emu 1 Lemma ddt of gauge action on a connection}u^{-1}\frac{\partial u}{\partial t}=-\frac{\partial s}{\partial t}s^{-1}=\chi.
 \end{equation}
 
The implication  ``$2\Longrightarrow 1$'' directly follows from \eqref{equ ddt of gauge action on a connection}.  
  \end{proof}
\subsection*{Variation of the Chern-Simons functional}
The formula for the gradient of the Chern-Simons functional is provided by the following. 
\begin{lem}\label{first variation of CS} In the setting of Definitions \ref{Def isotrivial} and \ref{Def CS}, suppose $H$ is closed, and let $A_{0}$ be an smooth connection on a Hermitian vector bundle $E\rightarrow Y$. The variation of the Chern-Simons functional \eqref{equ def CS functional} is given by the following. Suppose $a$ is a $C^{2}$  $\pi^{\star}adE-$valued $1-$form on $Y\times (-\epsilon,\epsilon)$, $\epsilon>0$, and $\frac{\partial a}{\partial t}|_{t=0}=v$. Then 
\begin{eqnarray}& &  \frac{d CS_{Y,H}(a)}{d t}|_{t=0}= 2\int_{Y}Tr(v\wedge F_{A_{0}+a}\wedge H)\label{equ 0 lem first variation of CS}
\\& = &\label{equ 1 lem first variation of CS}  \left\{ \begin{array}{cc} -2\int_{Y} \langle v,\star (F_{A_{0}+a}\wedge H)\rangle  dvol_{Y} & \textrm{when}\ dimY\ \textrm{is odd},  \\ 
2\int_{Y} \langle v,\star (F_{A_{0}+a}\wedge H)\rangle dvol_{Y} & \textrm{when}\ dimY\ \textrm{is even}.
\end{array}\right. 
\end{eqnarray}
\end{lem}
\begin{proof} This is absolutely standard. Since the integral formula \eqref{equ 1 lem first variation of CS} is a more than direct corollary of \eqref{equ 0 lem first variation of CS}, we only prove  \eqref{equ 0 lem first variation of CS}. We calculate 
\begin{eqnarray}\nonumber& &\frac{d}{dt}|_{t=0} Tr(a\wedge d_{A_{0}}a+\frac{2}{3}a\wedge a\wedge a+2a\wedge F_{A_{0}})
\\&= &Tr( v\wedge d_{A_{0}}a+ (d_{A_{0}}a)\wedge v-d_{A_{0}}(a\wedge v)+\frac{2}{3}[v\wedge a\wedge a+a\wedge v\wedge a+a\wedge a\wedge  v]\nonumber
\\& & +2v\wedge F_{A_{0}}). \nonumber
\\&= &Tr( 2v\wedge d_{A_{0}}a+2v\wedge a\wedge a+2v\wedge F_{A_{0}})-dTr(a\wedge v). \nonumber
\\&= &Tr( 2v\wedge F_{A_{0}+a})-dTr(a\wedge v).  \label{equ 1 Lemma first variation of CS}
\end{eqnarray}
Because $H$ is closed, the proof of \eqref{equ 0 lem first variation of CS} is complete by plugging \eqref{equ 1 Lemma first variation of CS} in the following 
\begin{equation}\frac{d}{dt}|_{t=0}  CS_{Y,H}=\int_{Y}\frac{d}{dt}|_{t=0}  \{Tr(a\wedge d_{A_{0}}a+\frac{2}{3}a\wedge a\wedge a+2a\wedge F_{A_{0}})\}\wedge H.
\end{equation}
\end{proof}

Using the variation formula \eqref{equ 0 lem first variation of CS}, in the following result, we have the invariance of the Chern-Simons functional along a smooth gauge orbit. This is crucial for the monotonicity of the functional along a gauge-modified ``gradient flow" like the instanton equation \eqref{equ condition 1 G2 instant on product}  below. 
\begin{lem}\label{lem CS constant along gauge orbit}(see \cite{Kim}) In the setting of Definitions \ref{Def isotrivial} and \ref{Def CS}, suppose $H$ is closed, and let $A$ be a connection on a Hermitian vector bundle $E\rightarrow Y$. For any smooth  gauge $s$ on $\pi^{\star}E\rightarrow Y\times I$, where $I$ is a bounded open  interval in $t$,  $\frac{d}{dt}CS_{Y,H}[s_{Y}(A)]=0$ in $I$. Consequently, $CS_{Y,H}$ is  constant along any smooth one-parameter gauge orbit.
\end{lem} 
\begin{proof} Let $\chi\triangleq \frac{ds}{dt}s^{-1}$, because $A$ is independent of $t$,  the identity \eqref{equ ddt of gauge action on a connection} yields\\ $\frac{\partial}{\partial t}[s_{Y}(A)]=s^{-1}(d_{Y,A}\chi)s$. 
Because conjugation by a unitary gauge preserves the inner-product, by the variation formula \eqref{equ 1 lem first variation of CS}, we calculate \begin{eqnarray}&& \frac{d}{dt}CS_{Y,H}[s_{Y}(A)]=(-1)^{n}2\int_{Y}\langle s^{-1}(d_{Y,A}\chi)s,\star(F_{s_{Y}(A)}\wedge H)\rangle dvol
\\& =& (-1)^{n}2\int_{Y}\langle d_{Y,A}\chi,\star(F_{A}\wedge H)\rangle =2\int_{Y}\langle \chi,\star d_{Y,A}(F_{A}\wedge H)\rangle dvol\nonumber
\\&=&0. \nonumber 
\end{eqnarray}\end{proof}

\subsection*{Proof of the first part of the main result}
Next we use the routine results established so far to prove Theorem \ref{Thm Rigidity}. We first calculate the $G_{2}$ and $Spin(7)-$instanton equations with respect to the splitting \eqref{equ splitting of curvature}. 

$G_{2}-$case: In the setting of Theorem \ref{Thm Rigidity}.$\mathbb{I}$, let $Y=X$ [the manifold with a $SU(3)-$structure]. The splitting \eqref{equ splitting of connection} reads $A=A_{X}+\chi dt$. Via  the splitting \eqref{equ splitting of curvature}, the $G_{2}-$instanton equation \eqref{equ def G2 instanton} is equivalent to 
\begin{equation}\label{equ 0 G2 case} F_{X,A_{X}}\lrcorner Re\Omega+J(\frac{\partial A_{X}}{\partial t}-d_{X,A_{X}}\chi)+(F_{X,A_{X}}\lrcorner_{\omega} \omega)dt=0, 
\end{equation}
where $J(\eta)\triangleq \eta\lrcorner_{\omega} \omega$ for an arbitrary $1-$form $\eta$. $J$ is the complex structure on $1-$forms, therefore $J^{2}=-Id$. Then 
\begin{equation}\label{equ G2 instanton on the product with J} F_{X,A_{X}}\lrcorner_{\omega} Re\Omega+J(\frac{\partial A_{X}}{\partial t}-d_{X,A_{X}}\chi)=0,\ F_{X,A_{X}}\lrcorner_{\omega} \omega=0.
\end{equation}
Applying $J$ to both sides of the first equation in \eqref{equ G2 instanton on the product with J}, using that $$J(F_{X,A_{X}}\lrcorner_{\omega} Re\Omega)=F_{X,A_{X}}\lrcorner_{\omega} Im\Omega, $$ we find 
\begin{equation}
F_{X,A_{X}}\lrcorner_{\omega} Im\Omega-(\frac{\partial A_{X}}{\partial t}-d_{X,A_{X}}\chi)=0. 
\end{equation}
Using $\star Re\Omega=Im\Omega$, we find $F_{X,A_{X}}\lrcorner_{\omega} Im\Omega=\star_{X}(F_{X,A_{X}}\wedge Re\Omega)$. Hence \eqref{equ G2 instanton on the product with J} [therefore the original instanton equation \eqref{equ def G2 instanton}] is equivalent to 
\begin{eqnarray}& &\label{equ condition 1 G2 instant on product}\frac{\partial A_{X}}{\partial t}=\star_{X}(F_{X,A_{X}}\wedge Re\Omega)+d_{X,A_{X}}\chi 
\\& &\label{equ condition 2 G2 instant on product} F_{X,A_{X}}\lrcorner_{\omega} \omega=0.
\end{eqnarray}

$Spin(7)-$case. In the setting of Theorem \ref{Thm Rigidity}.$\mathbb{II}$,  on the $8-$dimensional manifold $M\times S^{1}$, we still write the connection as $A=A_{M}+\chi dt$ [in view of \eqref{equ splitting of connection}]. Then we still have 
\begin{equation} F_{A}=F_{M,A_{M}}+(d_{M,A_{M}}\chi-\frac{\partial A_{M}}{\partial t})\wedge dt. 
\end{equation}
We recall that the orientation is  $dt\wedge \phi\wedge \psi$.

Purely algebraically, given a $2-$form $F$ on $\R^{8}=\R\times \R^{7}$,  we write $F=F_{ \R^{7}}+F_{0}\wedge e^{0}$, where $e^{0}$ stands for the coordinate vector of the $\R$ in the Cartesian product. Under the orientation $dt\wedge \phi_{Euc}\wedge \psi_{Euc}$, the algebraic equation $\star_{8}(F\wedge \Psi_{Euc})+F=0$ is equivalent to the following equations on $\R^{7}$. 
\begin{eqnarray}& &\star_{7}(F_{ \R^{7}}\wedge \psi_{Euc})=F_{0},\label{equ condition 1 Spin7 instant on product}
\\& & \star_{7}(F_{ \R^{7}}\wedge \phi_{Euc})+F_{ \R^{7}}=\star_{7}(\psi_{Euc}\wedge F_{0}).\label{equ condition 2 Spin7 instant on product}
\end{eqnarray}
Using the algebraic identity $(\theta\lrcorner \phi_{Euc})\lrcorner \phi_{Euc}=\star_{7}(\theta\wedge \phi_{Euc})+\theta$ for any $\theta\in \Lambda^{2}\R^{7}$, and contracting both hand sides of \eqref{equ condition 1 Spin7 instant on product} with $\phi_{Euc}$, we find that \eqref{equ condition 1 Spin7 instant on product} implies \eqref{equ condition 2 Spin7 instant on product}. This means  \eqref{equ condition 2 Spin7 instant on product} is  redundant.

Hence, on the manifold $M\times S^{1}$, the $Spin(7)-$instanton equation \eqref{equ def spin7 instanton equation} is equivalent to the following equation on $M$. 
\begin{equation}\label{equ Spin7 instanton on the product}\star_{\phi}(F_{M,A_{M}}\wedge \psi)=d_{M,A_{M}}\chi-\frac{\partial A_{M}}{\partial t}.
\end{equation}

\begin{proof}[\textbf{Proof of Theorem} \ref{Thm Rigidity} $\mathbb{I}1, \mathbb{I}2, \mathbb{II}1, \mathbb{II}2, \mathbb{III}1, \mathbb{III}2$:] We only fully prove the first 2 statements in  $\mathbb{I}$. The proof for (the first 2 statements in each of) $\mathbb{II},\ \mathbb{III}$ are the same.

The observation is that the instanton equation \eqref{equ condition 2 G2 instant on product}  can be considered as a  gauge-modified ``gradient flow" of the Chern-Simons functional with respect to $Re\Omega$. Then monotonicity of the functional forces the curvature term in \eqref{equ condition 2 G2 instant on product} to vanish. The half-closed condition for the $SU(3)-$structure (i.e. $Re\Omega$ is closed) corresponds to the closeness assumption on $H$ in Lemma \ref{first variation of CS}, \ref{lem CS constant along gauge orbit}.

A $G_{2}-$instanton  $A$  on $X\times S^{1}$  satisfies the system \eqref{equ condition 1 G2 instant on product}, \eqref{equ condition 2 G2 instant on product}. With respect to\\ $b\triangleq\star_{X}(F_{X,A_{X}}\wedge Re\Omega)$ and the $\chi$ in \eqref{equ condition 1 G2 instant on product},  let $s$ be the gauge on $X\times S^{1}$ produced by \eqref{equ0 Lemma ddt of gauge action on a connection} in Lemma \ref{lem ddt of gauge action on a connection}.$\mathbb{I} $. 
Identity \eqref{equ Lemma ddt of gauge action on a connection} says
 \begin{equation}\frac{\partial s_{X}(A_{X})}{\partial t}=\star_{X}(F_{s_{X}(A_{X})}\wedge Re\Omega). 
 \end{equation}

Hence the variation formula \eqref{equ 1 lem first variation of CS} yields the derivative of the Chern-Simons functional in $t$:
\begin{equation}\label{equ -1 proof of Thm Rigidity} \frac{dCS_{X,Re\Omega}[s_{X}(A_{X})]}{dt}=2\int_{X} |F_{s_{X}(A_{X})}\wedge Re\Omega|^{2}  dvol_{X}.
\end{equation}
We recall (from below the splitting  \eqref{equ splitting of connection}) the trivial fact that $A_{X}$ is smoothly periodic. We also observe that $s$ is a smooth gauge on $Y\times (-1,2\pi+1)$: because $\chi$ is smoothly periodic in $t$, the gauge $s$ produced by the ODE in Lemma \ref{lem ddt of gauge action on a connection}.$\mathbb{I}$ actually exists smoothly for all $t\in (-\infty,+\infty)$.  Via the invariance of Chern-Simons functional in Lemma \ref{lem CS constant along gauge orbit}, we obtain
\begin{eqnarray} & &CS_{X,Re\Omega}[s_{X}(A_{X})(0)]=CS_{X,Re\Omega}[A_{X}(0)]=CS_{X,Re\Omega}[A_{X}(2\pi)]\nonumber
\\&=&CS_{X,Re\Omega}[s_{X}(A_{X})(2\pi)]\label{equ 0 proof of Thm Rigidity}\ [\textrm{because}\ s(0)=Id],
\end{eqnarray}
where the invariance of the Chern-Simons functional in Lemma \ref{lem CS constant along gauge orbit} is only used for the last  among the $3$ equalities above. 
Integrating  \eqref{equ -1 proof of Thm Rigidity} over $t\in [0,2\pi]$, using \eqref{equ 0 proof of Thm Rigidity}, we find
\begin{eqnarray*}& &\label{equ 1 proof of Thm Rigidity}2\int_{0}^{2\pi}\int_{X} |F_{s_{X}(A_{X})}\wedge Re\Omega|^{2}  dvol_{X}dt
=CS_{X,Re\Omega}[s_{X}(A_{X})(2\pi)]-CS_{X,Re\Omega}[s_{X}(A_{X})(0)]
\\&=&0
\end{eqnarray*}
Therefore $F_{s_{X}(A_{X})}\wedge Re\Omega=0$ everywhere, which in turn implies that 
\begin{equation}\label{equ 2 proof of Thm Rigidity}F_{X,A_{X}}\wedge Re\Omega=0\ \textrm{over}\ X\times \{t\}\ \textrm{for any}\ t\in S^{1}.
\end{equation}
The condition \eqref{equ condition 2 G2 instant on product} and \eqref{equ 2 proof of Thm Rigidity} imply that $A_{X}(t)$ is Hermitian Yang-Mills with $0-$slope for all $t\in S^{1}$. In particular, the connection $A_{X}(0)$ on $Y$ is Hermitian Yang-Mills. This means that a $G_{2}-$instanton on $\pi^{\star}E\rightarrow X\times S^{1}$ yields a Hermitian Yang-Mills connection with $0-$slope on $E\rightarrow X$. On the other hand, the pullback of a Hermitian Yang-Mills with $0-$slope on $E\rightarrow X$ is a $G_{2}-$instanton on $X\times S^{1}$.  The proof of  Theorem \ref{Thm Rigidity}.$\mathbb{I}$.1 is complete.

 Next, we prove Theorem \ref{Thm Rigidity}.$\mathbb{I}$.2. Plugging the vanishing \eqref{equ 2 proof of Thm Rigidity} back into the \eqref{equ condition 1 G2 instant on product} for $A_{X}(t)$, we find 
\begin{equation}
\frac{\partial A_{X}}{\partial t}=d_{X,A_{X}}\chi. 
\end{equation}
Lemma \ref{lem ddt of gauge action on a connection}.$\mathbb{II}$ produces a $A_{X}(0)-$admissible broken gauge, and implies that $A=u[A_{X}(0)]$ is iso-trivial.  The proof of the ``only if'' in Theorem \ref{Thm Rigidity}.$\mathbb{I}$.2  is complete.

Given a Hermitian Yang-Mills connection $B$ with $0-$slope, for any $B-$admissible broken gauge $u$, $u(B)$ is a smooth connection on $X\times S^{1}$ (see Remark \ref{rmk isotrivial connections are smooth}). Because of the gauge invariance of the Hermitian Yang-Mills with $0-$slope,  $u(B)$ obviously satisfies the instanton equations  \eqref{equ condition 1 G2 instant on product} and \eqref{equ condition 2 G2 instant on product}.   The ``if'' in  Theorem \ref{Thm Rigidity}.$\mathbb{I}$.2 is proved.

The proof of Theorem \ref{Thm Rigidity}.$\mathbb{II}$ ($1$ and $2$) is by repeating exactly the above argument, changing  the manifold $X$ into the $7-$dimensional $M$,  changing the closed form $Re\Omega$ into the co-associative form $\psi$ on $M$, and using the $Spin(7)-$instanton equation \eqref{equ Spin7 instanton on the product} instead of the $G_{2}-$instanton equations \eqref{equ condition 1 G2 instant on product},  \eqref{equ condition 2 G2 instant on product}. 

 To prove Theorem \ref{Thm Rigidity}.$\mathbb{III}$  ($1$ and $2$),  by Kunneth-Theorem for the Hodge-DeRham cohomology and the condition that $H^{1}(X,\R)=0$, $H^{1}(X\times S^{1}, \R)$ is spanned by $dt$. Then on $X\times S^{1}$, $A$ is a projective $G_{2}-$instanton if and only if
 \begin{equation}\frac{\sqrt{-1}}{2\pi}\star(F_{A}\wedge \psi)=\mu dt\otimes  Id_{E}\ \textrm{for some real number}\ \mu.
\end{equation}

By the tensor calculations \eqref{equ 0 G2 case}--\eqref{equ condition 2 G2 instant on product}, $A$ is a projective $G_{2}-$instanton if and only if \eqref{equ condition 1 G2 instant on product} and $\frac{\sqrt{-1}}{2\pi}F_{X,A_{X}}\lrcorner_{\omega} \omega=\mu Id_{E}$ hold true [instead of the $0-$slope condition \eqref{equ condition 2 G2 instant on product}] . The rest of the proof is identical to that of Theorem \ref{Thm Rigidity}.$\mathbb{I}$ ($1$ and $2$) above. \end{proof}
\section{Topology of the moduli: proof of the second part of  Theorem \ref{Thm Rigidity} including I3, I4, II3, II4, III3, III4   \label{top}  }
\subsection*{The natural maps $\rho$, $\tau_{B}$, and $\tau$ between spaces of gauge equivalence classes of connections}
Let  $E\rightarrow Y$ be a Hermitian vector bundle as in Definition \ref{Def isotrivial}, and  $\mathfrak{M}^{isotrivial}_{Y\times S^{1},\pi^{\star}E}$ ($\mathfrak{M}^{isotrivial,irred}_{Y\times S^{1},\pi^{\star}E}$) denote the set of (irreducible) gauge equivalence classes of iso-trivial connections on\\ $\pi^{\star}E\rightarrow Y\times S^{1}$, respectively. The proof for the topological statements in the title of this section does not essentially involve the instanton or Hermitian Yang-Mills condition. 

Next, we define the natural map $\rho$ between $\mathfrak{M}^{isotrivial}_{Y\times S^{1},\pi^{\star}E}$ and $\Lambda_{E,Y}$ (see Definition \ref{Def top}). The ``$\rho$" in Theorem \ref{Thm Rigidity}.$\mathbb{I},\mathbb{II},\mathbb{III}.3$ is the restriction of the $\rho$ here onto the moduli of instantons. 

 \begin{Def}\label{Def map between isotrivial modulis} Let the map $\rho:\ \mathfrak{M}^{isotrivial}_{Y\times S^{1},\pi^{\star}E}\rightarrow \Lambda_{E,Y}$ be such that  $\rho(A)=A_{Y}(0)$. In other words, $\rho$ is the restriction of the component $A_{Y}$ to the zero $t-$slice.  For any $[B]\in \Lambda_{E,Y}$ and $B$ representing $[B]$, we define the fiber-wise map $\tau_{B}:\ \rho^{-1}([B])\rightarrow CON(\Gamma_{B})$ by 
\begin{equation}\label{equ Def map between isotrivial modulis 1}\tau_{B}\{[u(B)]\} =[u(2\pi)].\end{equation}
It is well defined because of the characterization of gauge equivalence in Proposition \ref{Prop gauge equivalence characterization}: as long as  $\widetilde{u}(B)$ is   gauge equivalent to $u(B)$,  $\widetilde{u}(2\pi)$ is conjugate to $u(2\pi)$ in the stabilizer group $\Gamma_{B}$.

  For any gauge $s$ on $Y$, the map $\gamma_{s}(b)\triangleq s^{-1}bs$ is an isomorphism from $\Gamma_{B}$ to $\Gamma_{s(B)}$ (as compact sub-groups of $\mathfrak{G}$). It degenerates to a homeomorphism from $CON(\Gamma_B)$ to 
 $CON[\Gamma_s(B)]$, which is still denoted by $\gamma_{s}(b)$. The following diagram commutes.   \begin{center}
\begin{tikzpicture}[->,>=stealth',shorten >=1pt,auto,node distance=2.5cm,
  thick,main node/.style={rectangle,draw}]
  \node at (0.5,0.80)   {$\tau_{s(B)}$} ;
  \node at (0.5,-0.75)   {$\tau_{B}$} ;
\node at (-1,0) (1)  {$\rho^{-1}([B])$} ;
     \node (2) at (2,1) {$CON[\Gamma_s(B)]$} ;
      \node (3) at (2,-1)  {$CON(\Gamma_B)$} ;
       \node  at (2.25,0)  {$\gamma_{s}$} ;
      \path[every node/.style={font=\sffamily\small}]
       (1) edge node [right] {} (2)
(1) edge node [right] {} (3)
(3) edge node [right] {} (2); \end{tikzpicture}
\end{center}

In a related manner, on irreducible connections, we define the map $$\tau:\ \mathfrak{M}^{isotrivial,irred}_{Y\times S^{1},\pi^{\star}E}\rightarrow Center[U(m)]\times \Lambda_{E,Y}^{irred}  $$ by 
\begin{equation}\label{equ Def map between isotrivial modulis 2}\tau\{[u(B)]\} = \{u(2\pi),[B]\}.\end{equation}
 \end{Def}
 \begin{rmk}\label{rmk maps are well defined} Similarly to the argument below \eqref{equ Def map between isotrivial modulis 1}, Proposition \ref{Prop gauge equivalence characterization} implies that 
 $\tau$ is also well defined i.e. it does not depend on the representative chosen in the gauge equivalence class $[u(B)]$.
 
 \end{rmk}

 \subsection*{Continuity of the natural maps $\rho$, $\tau_{B}^{-1}$, $\tau^{-1}$}
 After spelling out the definitions of $\rho,\ \tau_{B},\ \tau$, we now turn to continuity.   \begin{rmk}\label{Rmk Y and Y times} From here to  the end of the proof of Proposition \ref{prop hard continuity}, regarding the difference between the manifolds $Y\times S^{1}$ and $Y$ (cf. Remark \ref{rmk Y}), let $||\cdot ||$ be the norm (defined in \eqref{equ Def dist and norm on moduli}) on the product manifold $Y\times [0,2\pi]$ or $Y\times S^{1}$,  and let  $||\cdot||_{Y}$ mean the similar norm on the cross-section $Y$.  If there is only $Y$ but no $Y\times [0,2\pi]$ or $Y\times S^{1}$ in the context, we suppress the subscript $Y$ in the norm. 
 \end{rmk}

We start from the convergence of iso-trivial connections. 
\begin{lem}\label{Lem top up and down} In view of Definition \ref{Def ab gauges and iso trivial} and Remark \ref{Rmk Y and Y times}, suppose $B_{i}$, $B$ are smooth connections on $E\rightarrow Y$, and $[u_{i}(B_{i})],[u(B)]\in  \mathfrak{M}^{isotrivial}_{Y\times S^{1},\pi^{\star}E}$. Then\\   $\lim_{i\rightarrow \infty}d_{\Lambda_{\pi^{\star}E,Y\times S^{1}}}\{[u_{i}(B_{i})],[u(B)]\}=0$ if and only if there exist smooth gauges $g_{i}$ on\\ $\pi^{\star}E\rightarrow Y\times S^{1}$ such that
\begin{equation}\label{equ Lem top up and down}  \lim_{i\rightarrow \infty} ||B_{i}-\eta_{i,Y}(B)||=0\  \textrm{and}\ \lim_{i\rightarrow \infty} ||\eta_{i}^{-1}\frac{\partial \eta_{i}}{\partial t}||=0,\ \textrm{where}\ \eta_{i}\triangleq u g_{i}u_{i}^{-1}. \end{equation}
\end{lem}
\begin{proof} It suffices to observe that $||u_{i}(B_{i})-g_{i}[u(B)]||=||B_{i}-ug_{i}u_{i}^{-1}(B)||$.
 Then use\\ $B_{i}-\eta_{i}(B)=(B_{i}-B-\eta_{i}^{-1}d_{Y,B}\eta_{i})-\eta_{i}^{-1}\frac{\partial \eta_{i}}{\partial t}dt.$
\end{proof}

Lemma \ref{Lem top up and down} directly implies the continuity of the map $\rho$. This is crucial for Theorem \ref{Thm Rigidity}.$\mathbb{I},\mathbb{II},\mathbb{III}.3$. 
\begin{cor}\label{cor rho continuous} In the same setting as Definition \ref{Def map between isotrivial modulis}  and Lemma \ref{Lem top up and down},\\ $\rho:\ \mathfrak{M}^{isotrivial}_{Y\times S^{1},\pi^{\star}E}\rightarrow \Lambda_{E,Y}$ is continuous. Consequently, for any subset $\mathfrak{M}\subset \mathfrak{M}^{isotrivial}_{Y\times S^{1},\pi^{\star}E}$, under the induced topology, the map  $\rho:\ \mathfrak{M} \rightarrow \rho(\mathfrak{M})$ is continuous. 
\end{cor}

\begin{lem} \label{lem easy continuity} In the same setting as Definition \ref{Def map between isotrivial modulis}  and Lemma \ref{Lem top up and down}, \begin{enumerate}\item both $\tau_{B}$ and $\tau$ are bijective.  \item For any $[B]\in \Lambda_{E,Y}$ and representative $B$, $\tau_{B}^{-1}:\ CON(\Gamma_{B})\rightarrow \rho^{-1}([B])$ is continuous. 
\item $\tau^{-1}:\ Center[U(m)]\times \Lambda_{E,Y}^{irred} \rightarrow \mathfrak{M}^{isotrivial,irred}_{Y\times S^{1},\pi^{\star}E}$ is continuous. \end{enumerate}
\end{lem}
By the commutative diagram between \eqref{equ Def map between isotrivial modulis 1} and Remark \ref{rmk maps are well defined}, the continuity  in Lemma \ref{lem easy continuity}.2 is independent of the representative chosen in $[B]$. 

\begin{proof}  That $\tau_{B}$ is surjective follows directly from Lemma \ref{lem there exists a gauge connecting id to arbitrary a}. That  $\tau_{B}$ is injective  follows directly from  Proposition \ref{Prop gauge equivalence characterization}. By a similar argument, Proposition \ref{Prop gauge equivalence characterization} and  Lemmas \ref{lem there exists a gauge connecting id to arbitrary a} and  \ref{lem irred} imply that $\tau$ is a bijection. 

Next, we prove statement 2. Statement 3 follows by similar argument. 

Suppose $[a_{i}]\rightarrow [a]$ in $CON(\Gamma_{B})$. It means that there exist gauges $b_{i}\in \Gamma_{B}$ such that
\begin{equation}\label{equ 0 lem easy continuity} \lim_{i\rightarrow \infty}|| b_{i}^{-1}a_{i}b_{i}-a||_{Y}=0.\end{equation}
In view of the $C^{k}-$norm in Definition \ref{Def Cinfty top} below, because $d_{B}( b_{i}^{-1}a_{i}b_{i}-a)=0$ (and $B$ is smooth), for any $k\geq 0$ (particularly for $k=1$ which is all we need), we find 
\begin{equation}\label{equ 0.5 lem easy continuity} \lim_{i\rightarrow \infty}|| b_{i}^{-1}a_{i}b_{i}-a||_{C^{k}[Y,\pi^{\star}EndE]}=0.\end{equation}
As  in Claim \ref{clm there exists a good path} and below \eqref{eqn P commutes with conjugation}, let
 \begin{equation}u_{i}\triangleq P\{b_{i}^{-1}a_{i}b_{i}+\gamma(t)[Id-(b_{i}^{-1}a_{i}b_{i})]\},\ u=P[a+\gamma(t)(Id-a)],
\end{equation}
where $\gamma(t)$ avoids a small enough open neighborhood of all the roots of $det(a+x[Id-a])$ (in terms of $x$, see the material from Claim \ref{clm there exists a good path} to Claim \ref{clm det is a constant}). Then $u$ is a unitary gauge, and when $i$ is large enough (such that $b_{i}^{-1}a_{i}b_{i}+\gamma(t)[Id-(b_{i}^{-1}a_{i}b_{i})]$ is invertible), so is $u_{i}$. Moreover, \eqref{equ 0.5 lem easy continuity}  implies that $\lim_{i\rightarrow \infty}||u_{i}-u||_{C^{1}[Y\times [0,2\pi],\pi^{\star}EndE]}=0$ (see Definition \ref{Def Cinfty top} below). This implies $\lim_{i\rightarrow \infty}||u_{i}(B)-u(B)||=0$. Hence $$\lim_{i\rightarrow \infty}d_{\Lambda_{\pi^{\star}E,Y\times S^{1}}}([u_{i}(B)],[u(B)])=0.$$\end{proof}

\subsection*{Continuity of $\tau_{B}$ and $\tau$}
The continuity of   $\tau_{B}$ and $\tau$ is by another approach. 
\begin{prop} \label{prop hard continuity} In view of Lemma \ref{lem easy continuity}, \begin{enumerate}  \item $\tau_{B}:\ \rho^{-1}([B]) \rightarrow CON(\Gamma_{B}) $ is continuous for any $[B]\in \Lambda_{E,Y}$, therefore is a homeomorphism.
\item $\tau:\  \mathfrak{M}^{isotrivial,irred}_{Y\times S^{1},\pi^{\star}E}\rightarrow Center[U(m)]\times \Lambda_{E,Y}^{irred} $ is continuous, therefore is a homeomorphism.  \end{enumerate}
\end{prop}

\begin{proof} To prove ``1",   suppose \begin{equation}\label{equ 0 proof of Thm Rigidity top I3} u_{i}(B),\ u(B) \in \mathfrak{M}^{isotrivial}_{Y\times S^{1},\pi^{\star}E}\  \textrm{and}\  \lim_{i\rightarrow \infty}||u_{i}(B)-g_{i}[u(B)]||=0,\end{equation} where $g_{i}$ are  gauges on $Y\times S^{1}$. By definition of $\tau_{B}$, we need to show that 
\begin{equation}\label{equ 1 proof of Thm Rigidity top I3}\lim_{i\rightarrow \infty}d_{CON(\Gamma_{B})}\{[u_{i}(2\pi)], [u(2\pi)]\}=0\ [\textrm{see the metric in}\ \eqref{equ 0 Def top}].\end{equation}
In view of the equivalent conditions of convergence in Lemma \ref{Lem top up and down}, let $\eta_{i}$ be as in  \eqref{equ Lem top up and down}. Condition \eqref{equ 0 proof of Thm Rigidity top I3} yields 
\begin{equation}\label{equ 2 proof of Thm Rigidity top I3}\lim_{i\rightarrow \infty}||\eta_{i}^{-1}d_{B,Y}\eta_{i}||=0,\ \lim_{i\rightarrow \infty}||\eta_{i}^{-1}\frac{\partial \eta_{i}}{\partial t}||=0. \end{equation}

 By the existence in Lemma \ref{lem close to the stabilizer} below, and the first condition in \eqref{equ 2 proof of Thm Rigidity top I3}, there exists $a\in \Gamma_{B}$ such that 
 \begin{equation}\label{equ 2.5 proof of Thm Rigidity top I3}\lim_{i\rightarrow \infty}||\eta_{i}(0)-a||_{Y}=0.\end{equation}
 Integrating the second condition in \eqref{equ 2 proof of Thm Rigidity top I3} with respect to $t$, we find\\ $\lim_{i\rightarrow \infty}||\eta_{i}(2\pi)-\eta_{i}(0)||_{Y}=0$.  Then triangle-inequality yields 
\begin{equation}\label{equ 3 proof of Thm Rigidity top I3}
\lim_{i\rightarrow \infty}||\eta_{i}(2\pi)-a||_{Y}=0. 
\end{equation} 

Using \eqref{equ 2.5 proof of Thm Rigidity top I3}, \eqref{equ 3 proof of Thm Rigidity top I3}, $\eta_{i}(0)=g_{i}(0)=g_{i}(2\pi)$, and that  
\begin{eqnarray*}& &||a^{-1}u(2\pi)a-u_{i}(2\pi)||_{Y}=||u(2\pi)au^{-1}_{i}(2\pi)-a||_{Y}
\\&=&||u(2\pi)au^{-1}_{i}(2\pi)-u(2\pi)g_{i}(2\pi)u^{-1}_{i}(2\pi)+\eta_{i}(2\pi)-a||_{Y}
\\&\leq &||a-g_{i}(2\pi)||_{Y}+||\eta_{i}(2\pi)-a||_{Y},
\end{eqnarray*}
we find $\lim_{i\rightarrow \infty}||a^{-1}u(2\pi)a-u_{i}(2\pi)||_{Y}=0$. Therefore \eqref{equ 1 proof of Thm Rigidity top I3} is true. The proof of ``1" is complete. 

Next, on irreducible connections, we prove ``$2$" similarly to the fiber-wise case above. Suppose
\begin{equation}\label{equ 0 proof of Thm Rigidity top II3}\lim_{i\rightarrow \infty}||u_{i}(B_{i})-g_{i}[u(B)]||=0,\ [\textrm{note the slight difference from}\ \eqref{equ 0 proof of Thm Rigidity top I3}].  \end{equation} By definition of $\tau$, we need to show 
\begin{equation}\label{equ 1 proof of Thm Rigidity top II3}\lim_{i\rightarrow \infty}d_{\Lambda_{E,Y}}([B_{i}],[B])=0\ \textrm{and}\ \lim_{i\rightarrow \infty}||u_{i}(2\pi)-u(2\pi)||_{Y}=0.
\end{equation}
We now re-state the two identities given by Lemma \ref{Lem top up and down} under condition \eqref{equ 0 proof of Thm Rigidity top II3}. Namely, let $\eta_{i}$ be as in  \eqref{equ Lem top up and down},  Lemma \ref{Lem top up and down} yields the following.
\begin{eqnarray}& &\label{equ 1.5 proof of Thm Rigidity top II3} \lim_{i\rightarrow \infty}||B_{i}-\eta_{i,Y}(B)||=0\ (\textrm{note}\ B_{i}-\eta_{i,Y}(B)= B_{i}-B-\eta_{i}^{-1}d_{B,Y}\eta_{i}),\\
& &\label{equ 2 proof of Thm Rigidity top II3} \lim_{i\rightarrow \infty}||\eta_{i}^{-1}\frac{\partial \eta_{i}}{\partial t}||=0.\end{eqnarray}The first desired condition in \eqref{equ 1 proof of Thm Rigidity top II3} is directly implied by \eqref{equ 1.5 proof of Thm Rigidity top II3}. It remains to prove  the second using  irreducibility and  \eqref{equ 2 proof of Thm Rigidity top II3}.  Again, integrating \eqref{equ 2 proof of Thm Rigidity top II3} with respect to $t$, we find $\lim_{i\rightarrow \infty}||\eta_{i}(0)-\eta_{i}(2\pi)||_{Y}= 0$.  Hence
\begin{equation}\label{equ 3 proof of Thm Rigidity top II3}
\lim_{i\rightarrow \infty}||g_{i}(0)-u(2\pi)g_{i}(2\pi)u_{i}^{-1}(2\pi)||_{Y}=0.
\end{equation}
Irreducibility implies that $u(2\pi),\ u_{i}(2\pi)\in Center[U(m)]$. Using \eqref{equ 3 proof of Thm Rigidity top II3}
 and that
\begin{eqnarray*}
& &||u_{i}(2\pi)-u(2\pi)||_{Y}=||Id-u(2\pi)u_{i}^{-1}(2\pi)||_{Y}= ||g_{i}(0)-g_{i}(0)u(2\pi)u_{i}^{-1}(2\pi)||_{Y}
\\& =& ||g_{i}(0)-u(2\pi)g_{i}(2\pi)u_{i}^{-1}(2\pi)||_{Y}\ [\textrm{using}\ g_{i}(0)=g_{i}(2\pi)\ \textrm{and}\ u(2\pi) g_{i}(0)=g_{i}(0) u(2\pi)],
\end{eqnarray*}
we find $\lim_{i\rightarrow \infty}||u_{i}(2\pi)-u(2\pi)||_{Y}=0$. Hence the second desired condition in \eqref{equ 1 proof of Thm Rigidity top II3} holds. \end{proof}
\subsection*{Proof of the other part of Theorem \ref{Thm Rigidity} and of the example}
The above facts can be  assembled into the following proof. 
 \begin{proof}[\textbf{Proof of Theorem} \ref{Thm Rigidity} $\mathbb{I}3-4,  \mathbb{II}3-4, \mathbb{III}3-4$ :] We only show $\mathbb{I}3-4$, the others are the same.  Theorem \ref{Thm Rigidity}.$\mathbb{I}.2$ means $\mathfrak{M}_{X\times S^{1},\pi^{\star}E,\phi}\subset \mathfrak{M}^{isotrivial}_{X\times S^{1},\pi^{\star}E}$. Still by  Theorem \ref{Thm Rigidity}.$\mathbb{I}.2$,  restricting the $\rho$ in Definition \ref{Def map between isotrivial modulis} to $\mathfrak{M}_{X\times S^{1},\pi^{\star}E,\phi}$, we obtain  $\rho(\mathfrak{M}_{X\times S^{1},\pi^{\star}E,\phi})=\mathfrak{M}_{X,E,\omega-HYM-0}$. The continuity of $\rho$ follows directly from Corollary \ref{cor rho continuous} (restricted to the moduli $\mathfrak{M}_{X\times S^{1},\pi^{\star}E,\phi}$ of instantons). The second statement in $\mathbb{I}3$ follows from the topological type of a fiber characterized in Proposition \ref{prop hard continuity}.$1$ (applied to an arbitrary $[B]\in \mathfrak{M}_{X,E,\omega-HYM-0}$). 
 
 Similarly, by Lemma \ref{lem irred} (on irreducibility) and $\mathbb{I}2$,    $\mathfrak{M}^{irred}_{X\times S^{1},\pi^{\star}E,\phi}=\rho^{-1}( \mathfrak{M}^{irred}_{X,E,\omega-HYM-0})$. Then $\mathbb{I}4$ follows from Proposition \ref{prop hard continuity}.$2$ restricted to the moduli $\mathfrak{M}^{irred}_{X\times S^{1},\pi^{\star}E,\phi}$ of irreducible instantons. \end{proof}

\begin{proof}[\textbf{Proof of Corollary} \ref{cor example Thomas}:] By \cite[Theorem 4.8 and page 418 Example 1]{Thomas}, there is  a bundle $E\rightarrow X_{CY}$ as in  Corollary \ref{cor example Thomas} and  a K\"ahler-class $[\underline{\omega}]$ such that the following holds.
\begin{itemize} \item $\mathfrak{M}^{AG}_{X_{CY},E,[\underline{\omega}]-stable}$ consists of one point.
\item Any poly-stable holomorphic structure on $E$ is stable, therefore simple. By \cite[VII Proposition 4.14]{Kobayashi} and the Donaldson-Uhlenbeck-Yau Theorem  (stated in Definition \ref{Def DUY}), we obtain 
\begin{equation}\label{equ 0 Proof of Cor} \mathfrak{M}^{irred}_{X_{CY},E,\omega-HYM}=\mathfrak{M}_{X_{CY},E,\omega-HYM},\ \textrm{and both of them consist of one point}.
\end{equation}
\end{itemize} 
Let $\Omega_{0}$ be a trivialization of $K_{X_{CY}}$.  There exists $c_{0}\in \C$ (unique up a unitary factor) such that $\Omega\triangleq c_{0}\Omega_{0}$ satisfies 
\begin{equation} \int_{X_{CY}}\frac{\underline{\omega}^{3}}{3!}=\frac{\sqrt{-1}}{8}\int_{X_{CY}}\Omega\wedge\bar{\Omega}=\frac{1}{4}\int_{X_{CY}}Re\Omega\wedge Im\Omega.\end{equation}
Under the above integral normalization condition, Yau \cite{Yau} showed that there exists a unique $\omega\in [\underline{\omega}]$ satisfying the point-wise volume-form equation in Definition \ref{Def CY quadruple}.3.  The proof is then complete by \eqref{equ 0 Proof of Cor} and Theorem \ref{Thm Rigidity}.$\mathbb{III}. 3,4$. \end{proof}

 \section{Appendix}
 \subsection*{Existence of a normalized Hermitian metric in any conformal class on a $6-$manifold with a nowhere vanishing $(3,0)-$form}
The following Lemma produces $6$-manifolds with $SU(3)-$structures. It helps us produce half-closed ones  and makes our main result more meaningful (see Remark \ref{Rmk abundant SU(3) str}). Moreover, the point-wise frame in \eqref{equ lem unitary frame} helps the tensor calculations related to the instanton equations (see from \eqref{equ 0 G2 case} to \eqref{equ condition 2 G2 instant on product}).\begin{lem}\label{lem unitary frame} Let $X$ be a closed $6-$dimensional manifold  with an almost complex structure $J$ and a nowhere vanishing $(3,0)-$form $\Omega$. For any conformal class  $[\underline{g}]$ of Hermitian metrics, there is a unique Hermitian metric $g$  such that $|\Omega|_{g}^{2}=8$ i.e. $\frac{\omega^{3}}{3!}=\frac{1}{4}Re\Omega\wedge Im\Omega$, where \\
$\omega\triangleq g(J\cdot,\cdot)$ is the associated $(1,1)-$form of $g$. Consequently, at an arbitrary point $p$, there exists a unitary frame $v^{1},\ v^{2},\ v^{3} \in T^{1,0}_{p}(X)$ with respect to $g$ such that 
 \begin{equation}\label{equ lem unitary frame} \omega |_{p}=\frac{\sqrt{-1}}{2}\Sigma_{i=1}^{3}v^{i}\wedge \bar{v}^{i},\ \Omega |_{p}=v^{1}\wedge v^{2}\wedge v^{3}. 
 \end{equation}
 \end{lem}
 \begin{proof} Let $\underline{\omega}$ be the positive $(1,1)-$form associated to the representative $\underline{g}$ of the conformal class. For any $p$, let $u^{1}, u^{2}, u^{3} \in T^{1,0}_{p}(X)$ be a unitary frame such that $\underline{\omega}=\frac{\sqrt{-1}}{2}\Sigma_{i=1}^{3}u^{i}\wedge \bar{u}^{i}$ and $\Omega=c_{0}(u^{1}\wedge u^{2}\wedge u^{3})$.
 Then $|c_{0}|^{2}=\frac{|\Omega|_{\underline{g}}^{2}}{8}$ is smooth. Let $h^{i}\triangleq |c_{0}|^{\frac{1}{3}}u^{i}$, we find 
 \begin{equation}\Omega=c_{1}(h^{1}\wedge h^{2}\wedge h^{3}),\  c_{1}=\frac{c_{0}}{|c_{0}|},\ \textrm{thus}\ |c_{1}|=1.\ \textrm{We define}\ \omega\triangleq |c_{0}|^{\frac{2}{3}}\underline{\omega}=\frac{\sqrt{-1}}{2}\Sigma_{i=1}^{3}h^{i}\wedge \bar{h}^{i}.
 \end{equation}
 This means that $|\Omega|^{2}_{g}=8$, where $g=\omega(\cdot,J\cdot )$ is the corresponding Hermitian metric. Finally, let $c_{2}$ be an arbitrary cubic root of $c_{1}$ at $p$, and $v_{i}\triangleq c_{2}h_{i}$. The existence of the unitary frame at $p$ (in Lemma \ref{lem unitary frame}) is proved. 
 
Next, in a fixed conformal class, we show the uniqueness  of the Hermitian metric which satisfies $|\Omega|^{2}_{g}=8$.  Suppose $\widetilde{g}=e^{2f}g$ is a another Hermitian metric satisfying \eqref{equ lem unitary frame} everywhere, then 
 \begin{equation}8=|\Omega|^{2}_{\widetilde{g}}=e^{-6f}|\Omega|^{2}_{g}=8e^{-6f}\Longrightarrow f=0.
\end{equation}
The uniqueness is proved. 
\end{proof}

 \subsection*{Smooth sections of the pullback bundle over $Y\times [0,2\pi]$}
 The following definition of smooth sections (of the pullback bundle) on the manifold $Y\times [0,2\pi]$ with boundary is applied to iso-trivial connections and related places. Please see Definition \ref{Def ab gauges and iso trivial} for example. In practice, the dummy notation ``$E$" below might be the endomorphism bundle of a specific $E$. 
 \begin{Def}\label{Def Cinfty top}In conjunction with the finite open cover (coordinate chart of the bundle) as part of Definition \ref{Def isotrivial} of a Hermitian vector bundle,  let $||\cdot||_{C^{k}[Y, E]}$ denote the $C^{k}-$norm of a section of the bundle $E\rightarrow Y$. It is defined as the weighted sum of the $C^{k}-$norms of the matrix-valued functions in coordinate charts with respect to the partition of unity.

 We define the $C^{\infty}[Y, E]-$topology by the following.
\begin{equation} \lim_{j\rightarrow \infty} \phi_{j}=\phi_{\infty}\ \textrm{in}\ C^{\infty}\Longleftrightarrow \lim_{j\rightarrow \infty} \phi_{j}=\phi_{\infty}\ \textrm{in}\ C^{k}[Y, E]\ \textrm{for every}\ k. 
\end{equation}
This is a metric topology by \cite[Section 1.46]{RudinFun}. 

Although we can use the $C^{k}-$norm with respect to a fixed open cover, it is heuristic to make the following remark. Because $Y$ is compact so  there are finite covers, the $C^{k}-$norms defined by different finite open covers (with the associated trivializations) are equivalent. 

Suppose $s$ is a continuous section of $\pi^{\star}E\rightarrow Y\times [0,2\pi]$ which is smooth on $Y\times (0,2\pi)$. $s$ is said to be smooth on $\pi^{\star}E\rightarrow Y\times [0,2\pi]$ if under the $C^{\infty}[Y,E]-$topology, for any $k\geq 0$, both $\lim_{t\rightarrow 0}\frac{\partial ^{k}s}{\partial t^{k}}$ and $\lim_{t\rightarrow 2\pi}\frac{\partial ^{k}s}{\partial t^{k}}$ exist. 
 Then for any $k\geq0$,  $\frac{\partial ^{k}s}{\partial t^{k}}$ extends continuously to $Y\times [0,2\pi]$. The values at the end points are still denoted by $\frac{\partial ^{k}s}{\partial t^{k}}(0)$ and  $\frac{\partial ^{k}s}{\partial t^{k}}(2\pi)$ respectively. 
 The $C^{k}-$norm on $Y\times [0,2\pi]$ is defined naturally as 
 \begin{equation} ||s||_{C^{k}\{Y\times [0,2\pi],\pi^{\star}E\}}\triangleq \sup_{0\leq i+j\leq k,\ t_{0}\in [0,2\pi]} ||\frac{\partial^{i}s}{\partial t^{i}}(t_{0})||_{C^{j}[Y,E]}.
 \end{equation}

 A smooth connection on $\pi^{\star}E\rightarrow Y\times [0,2\pi]$ is defined similarly. 
 \end{Def}
 \subsection*{Analyticity of the square root function of positive Hermitian matrices}
We now turn to the analyticity of the matrix square root function. This  is applied in Lemma \ref{lem there exists a gauge connecting id to arbitrary a} to show that the broken gauge is smooth. For lack of reference, we still give the full proof. 

\begin{lem}\label{lem sqrt is real analytic}In view of Claim \ref{clm unique square root}, the map $\sqrt{\cdot}\ :Herm^{+}_{m\times m}\rightarrow Herm^{+}_{m\times m}$ is real-analytic. 
\end{lem}
\begin{proof} The idea is to interpret $\sqrt{.}$ as an implicit function, then use the implicit function theorem. 
We consider $F(H,h)\triangleq H-h^{2}:\ Herm^{+}_{m\times m}\oplus Herm^{+}_{m\times m}\rightarrow Herm_{m\times m}$. For any $H_{0},\ h_{0}$ such that  $F(H_{0},h_{0})=0$, it suffices to show that the linearization $L_{h,(H_{0},h_{0})}: Herm_{m \times m}\rightarrow Herm_{m\times m}$ with respect to $h$ is invertible. We calculate 
\begin{equation} -L_{h,(H_{0},h_{0})}g=h_{0}g+gh_{0},\ \textrm{where}\ g \ \textrm{is the variation of}\  h. 
\end{equation}
Suppose 
\begin{equation}h_{0}g+gh_{0}=0. 
\end{equation}
For any eigenvalue $\mu$ of $g$, let $v$ be a corresponding eigenvector. Because $h_{0}$ and $g$ are both Hermitian, $\mu$ must be real,  and we compute
\begin{equation*}0=(h_{0}gv,v)+(gh_{0}v,v)=2\mu(h_{0}v,v)\ \textrm{where}\ ``(\cdot,\cdot)'' \ \textrm{is the Euclidean Hermitian product}. 
\end{equation*}
Because $h_{0}$ is positive definite Hermitian,  $(h_{0}v,v)  >0$. Then $\mu=0$. Because $\mu$ is an arbitrary eigenvalue of $g$,  $g$ must be $0$. Therefore $KerL_{h,(H_{0},h_{0})}=\{0\}$, and $L_{h,(H_{0},h_{0})}$ is an linear isomorphism from $Herm_{m\times m}$ to itself. 

By the analytic implicit function theorem (see \cite[Page 1081]{Whittlesey}), and the uniqueness of the square root, $h(H)=\sqrt{H}$ is real-analytic near $H_{0}$. Because $H_{0}$ is arbitrary, the proof is complete. 
\end{proof}

\subsection*{Producing a stabilizer from an ``approximately parallel" gauge}
Briefly speaking, the following result provides the limit stabilizer $a$ in \eqref{equ 2.5 proof of Thm Rigidity top I3}. 
 \begin{lem}\label{lem close to the stabilizer}In the setting of Proposition \ref{prop hard continuity}, for any $\epsilon>0$, there is a $\delta$ with the following property. Suppose $\eta$ is a gauge on $E\rightarrow Y$ and  $||\eta^{-1}d_{B}\eta||<\delta$. Then there is an $a\in \Gamma_{B}$ such that $d_{B}a=0\, \textrm{and}\ ||\eta-a||<\epsilon$. 
   \end{lem}
   \begin{proof}If not, there is an $\epsilon>0$ and a sequence $\eta_{j}$ such that 
\begin{equation}\label{equ 0 lem close to the stabilizer}||\eta_{j}^{-1}d_{B}\eta_{j}||\rightarrow 0,\end{equation}
 but  for any $j$, $||s-\eta_{j}||<\epsilon\Longrightarrow d_{A}s\neq 0$.  
 
Because $\eta_{j}$ is unitary, it is bounded. Then  \eqref{equ 0 lem close to the stabilizer} implies that $||\eta_{j}||+||d_{B}\eta_{j}||\leq C$, where $C$ is independent of $j$. Moreover, $||d_{B}\eta_{j}||\rightarrow 0$. Then Arzela-Ascoli Theorem implies that $\eta_{j}$ sub-converges uniformly to $a$ (in $C^{0}[Y, End(E)]$).
 \begin{clm}\label{clm lem close to the stabilizer} At any point $p\in Y$, and under any coordinate chart of $EndE$,  $a$ admits all (first order) partial derivatives, and $d_{B}a=0$ at $p$ (defined only by partial derivatives of $a$). \end{clm}
 Assuming the above claim, because the connection $B$ is smooth, the condition $d_{B}a=0$ implies by definition that  $a$ is smooth. Thus $a\in \Gamma_{B}$. This is a contradiction to the line below \eqref{equ 0 lem close to the stabilizer}. 
 
It remains to prove Claim \ref{clm lem close to the stabilizer}. It is completely elementary.  The idea is very simple:   in the coordinate chart, the endomorphisms $a$ and $\eta_{i}$ are matrix-valued functions.  On any line segment passing through $p$, apply the classical fact \cite[Theorem 7.17]{Rudin} on single variable calculus. This classical fact  says that suppose a sequence of functions  on the interval $[a,b]$ converges at one point, and the sequence of derivatives converges uniformly. Then the sequence converges uniformly, and the derivative of the limit exists and is equal to the limit of the sequence of derivatives. 

We now carry out the elementary detail.  Under a coordinate chart for $EndE$, let $l_{p}(t)$ be a closed line segment through $p$, $t\in [-1,1]$. The condition \eqref{equ 0 lem close to the stabilizer} on covariant derivative of $\eta_{j}$ and the  condition that $\eta_{j}$ converges uniformly imply that on the line segment $l_{p}(t)$,  $\frac{d\eta_{j}}{dt}$ converges uniformly. This means the two conditions in \cite[Theorem 7.17]{Rudin} are satisfied.  Then it says that
$a$  is differentiable in $t$, and $\lim_{j\rightarrow \infty} \frac{d\eta_{j}}{dt}=\frac{da}{dt}$. Because the line segment is arbitrary, this means $a$ admits all partial derivatives (under the coordinate chart) at an arbitrary point $p$. 

The condition that $d_{B}\eta_{j}\rightarrow 0$ uniformly implies that on the line segment,  $\frac{d\eta_{j}}{dt}+[B(\dot{l}_{p}),\eta_{j}]$ tends to $0$ uniformly. Because $\frac{d\eta_{j}}{dt}$ converges uniformly to $\frac{da}{dt}$,  and $\eta_{j}$ converges uniformly to $a$,  we find $\frac{da}{dt}+[B(\dot{l}_{p}),a]=0$. This implies  $\nabla_{B,\dot{l}_{p}} a=0$ at $p$.  Again, because $p$ and $l_{p}$ are arbitrary,  $d_{B}a=0$ everywhere. The proof of Claim \ref{clm lem close to the stabilizer} is complete. \end{proof}


\small
\begin{thebibliography}{0}
\bibitem{Berger}M. Berger.  \emph{Sur les groupes d'holonomie homogènes des variétés a connexion affines et des variétés riemanniennes}. Bull. Soc. Math. France (1953), 83: 279-330,
\bibitem{Earp}O. Calvo-Andrade, L.O. Rodr\'iguez D\'iaz, H.N. S\'a Earp. \emph{Gauge Theory and $G_{2}-$Geometry on Calabi-Yau Links}. arxiv1606.09271 
 \bibitem{FB} F. Bagagiolo.  \emph{Ordinary Differential Equation}. Dipartimento di Matematica, Universit\`a di Trento. science.unitn.it
 \bibitem{Burke}J.V. Burke.   $https://sites.math.washington.edu/~burke/crs/555/555-notes/exist.pdf$
 \bibitem{Kim}Y.H. Byun, J.H. Kim. \emph{Cohomology of Flat Bundles and a Chern-Simons Functional}. arxiv14096569v1. 
\bibitem{Harland} B. Charbonneau, D. Harland. \emph{Deformations of nearly K\"ahler instantons}. Commun. Math. Phys. 348, 959--990 (2016). 
  \bibitem{Donaldsonpaper1}S.K. Donaldson. \emph{Anti-self-dual Yang-Mills connections over complex algebraic surfaces and stable vector bundles}. Proc. LMS 50 (1985) 1-26.
  \bibitem{Donaldsonpaper2}S.K. Donaldson. \emph{Infinite determinants, stable bundles and curvature}. Duke Math. J. 54 (1987), no. 1, 231-247.
 \bibitem{Donaldson} S.K. Donaldson, P.B. Kronheimer. \emph{The Geometry of Four-Manifolds}. Oxford Mathematical Monographs. 1990.  
 \bibitem{DonaldsonSegal} S.K. Donaldson, E. Segal.
\emph{Gauge Theory in Higher Dimensions, II}. from: Geometry
of special holonomy and related topics (NC Leung, ST
Yau, editors), Surv. Differ.
Geom. 16, International Press (2011) 1--41.
 \bibitem{DonaldsonThomas} S.K. Donaldson, R.P. Thomas. \emph{Gauge Theory in Higher Dimensions}. from: The Geometric Universe. Oxford Univ. Press (1998) 31--47.
\bibitem{Jardim} M. Jardim.  \emph{Stable bundles on 3-fold hypersurfaces}. Bull Braz Math Soc, New Series 38(4), 649-659.
\bibitem{Joyce}D. Joyce.  \emph{Conjectures on counting associative 3-folds in G2-manifolds}. arXiv:1610.09836
 \bibitem{KMT}S. Karigiannis, B. Mckay, M.P. Tsui. \emph{Soliton solutions for the Laplacian co-flow of some $G_{2}-$structures with symmetry}. Differential Geometry and its Applications 30 (2012) 318-333. 
 \bibitem{Kobayashi} S. Kobayashi. \emph{Differential Geometry of Complex Vector Bundles}. Publication of the Mathematical Society of Japan. 1987.

\bibitem{Menet} G. Menet, J. Nordstr\"om, H. S\'a Earp. \emph{Construction of $G_2$-instantons via twisted connected sums}. arXiv:1510.03836 
\bibitem{RudinFun}W. Rudin. \emph{Functional Analysis}. McGraw-Hill. 1973
\bibitem{Rudin}W. Rudin. \emph{Principles of Mathematical Analysis}. 3rd Edition. International Series of Pure and Applied Mathematics. 1976

\bibitem{Simons}J. Simons. \emph{On the transitivity of holonomy systems}. Annals of Mathematics (1962) 76 (2). 213--234.
\bibitem{Thomas} R. Thomas. \emph{A Holomorphic Casson Invariant For Calabi-Yau 3-folds, and Bundles on K3 Fibration}. J. Differential. Geometry. 53 (1999). 367--438. 
\bibitem{UY} K. Uhlenbeck, S.T. Yau. \emph{On the existence of Hermitian Yang-Mills-Connections on Stable Bundles over K\"ahler manifolds}. Comm. Pure. Appl. Math. 39 (1986) 257-293
\bibitem{Walpuski} T. Walpuski. \emph{Gauge theory on $G_{2}-$manifolds}. Thesis presented for the degree of Doctor of Philosophy. Imperial College London. 2013.
\bibitem{Whittlesey}E.F. Whittlesey. \emph{Analytic Functions in Banach Spaces}. Proceedings of the American Mathematical Society
Vol. 16, No. 5 (Oct. 1965), pp. 1077-1083

\bibitem{Yau} S.T. Yau. \emph{On the Ricci curvature of a compact K\"ahler manifold and the Complex Monge
Amp'ere equation I}. Comm. Pure Appl. Math. 31 (1978).
\end{thebibliography}
\end{document}